\documentclass[11pt]{amsart}
\usepackage[a4paper]{geometry}
\usepackage{amstext}
\usepackage{amsthm}
\usepackage{amssymb}
\usepackage{setspace}
\usepackage{wasysym}
\usepackage[unicode=true,pdfusetitle,
 bookmarks=true,bookmarksnumbered=false,bookmarksopen=false,
 breaklinks=false,pdfborder={0 0 1},backref=false,colorlinks=false]
 {hyperref}

\makeatletter

\ProvideTextCommandDefault{\guillemotleft}{%
  {\usefont{U}{lasy}{m}{n}\char'50\kern-.15em\char'50}%
\penalty10000\hskip0pt\relax%
}
\ProvideTextCommandDefault{\guillemotright}{%
  \penalty10000\hskip0pt%
  {\usefont{U}{lasy}{m}{n}\char'51\kern-.15em\char'51}%
}

\numberwithin{equation}{section}
\numberwithin{figure}{section}
\theoremstyle{plain}
\newtheorem{thm}{\protect\theoremname}[section]
\theoremstyle{remark}
\newtheorem{rem}[thm]{\protect\remarkname}
\theoremstyle{definition}
\newtheorem{defn}[thm]{\protect\definitionname}
\theoremstyle{plain}
\newtheorem{cor}[thm]{\protect\corollaryname}
\theoremstyle{definition}
\newtheorem{example}[thm]{\protect\examplename}
\theoremstyle{plain}
\newtheorem{question}[thm]{\protect\questionname}
\theoremstyle{remark}
\newtheorem*{acknowledgement*}{\protect\acknowledgementname}
\theoremstyle{plain}
\newtheorem{lem}[thm]{\protect\lemmaname}
\theoremstyle{plain}
\newtheorem{prop}[thm]{\protect\propositionname}
\theoremstyle{remark}
\newtheorem{claim}[thm]{\protect\claimname}

\usepackage{amsthm}
\usepackage{bbm}

\theoremstyle{plain}

\makeatother

\providecommand{\acknowledgementname}{Acknowledgement}
\providecommand{\claimname}{Claim}
\providecommand{\corollaryname}{Corollary}
\providecommand{\definitionname}{Definition}
\providecommand{\examplename}{Example}
\providecommand{\lemmaname}{Lemma}
\providecommand{\propositionname}{Proposition}
\providecommand{\questionname}{Question}
\providecommand{\remarkname}{Remark}
\providecommand{\theoremname}{Theorem}

\begin{document}
\title{Integrability of pushforward measures by analytic maps}
\author{Itay Glazer}
\address{Department of Mathematics, University of Oxford, Andrew Wiles Building, Radcliffe Observatory
Quarter (550), Woodstock Road, Oxford, OX2 6GG, UK and Department of Mathematics, Technion – Israel Institute of Technology, Haifa 3200003, Israel}
\email{itayglazer@gmail.com}
\urladdr{https://sites.google.com/view/itay-glazer}
\author{Yotam I. Hendel}
\address{KU Leuven, Department of Mathematics, B-3001 Leuven, Belgium and Ben Gurion University of the Negev,  Department of Mathematics, Be’er Sheva 8410501, Israel}
\email{yotam.hendel@gmail.com}
\urladdr{https://sites.google.com/view/yotam-hendel}
\author{Sasha Sodin}
\address{School of Mathematical Sciences, Queen Mary  University of London, Mile End Road, London E1 4NS, UK  and Einstein Institute of Mathematics, The Hebrew University of Jerusalem, Givat Ram, Jerusalem 91904, Israel}
\email{a.sodin@qmul.ac.uk}
\urladdr{https://webspace.maths.qmul.ac.uk/a.sodin/}
\begin{abstract}
Given a map $\phi:X\rightarrow Y$ between $F$-analytic manifolds
over a local field $F$ of characteristic $0$, we introduce an invariant
$\epsilon_{\star}(\phi)$ which quantifies the integrability of pushforwards
of smooth compactly supported measures by $\phi$. We further define
a local version $\epsilon_{\star}(\phi,x)$ near $x\in X$. These
invariants have a strong connection to the singularities of $\phi$.

When $Y$ is one-dimensional, we give an explicit formula for $\epsilon_{\star}(\phi,x)$,
and show it is asymptotically equivalent to other known singularity
invariants such as the $F$-log-canonical threshold $\operatorname{lct}_{F}(\phi-\phi(x);x)$
at $x$.

In the general case, we show that $\epsilon_{\star}(\phi,x)$ is bounded
from below by the $F$-log-canonical threshold $\lambda=\operatorname{lct}_{F}(\mathcal{J}_{\phi};x)$
of the Jacobian ideal $\mathcal{J}_{\phi}$ near $x$. If $\dim Y=\dim X$,
equality is attained. If $\dim Y<\dim X$, the inequality can be strict;
however, for $F=\mathbb{C}$, we establish the upper bound $\epsilon_{\star}(\phi,x)\leq\lambda/(1-\lambda)$,
whenever $\lambda<1$.

Finally, we specialize to polynomial maps $\varphi:X\rightarrow Y$
between smooth algebraic $\mathbb{Q}$-varieties $X$ and $Y$. We
geometrically characterize the condition that $\epsilon_{\star}(\varphi_{F})=\infty$
over a large family of local fields, by showing it is equivalent to
$\varphi$ being flat with fibers of semi-log-canonical singularities. 
\end{abstract}

\maketitle
\pagenumbering{arabic}

\global\long\def\Q{\mathbb{Q}}%
\global\long\def\N{\mathbb{N}}%
\global\long\def\R{\mathbb{\mathbb{R}}}%
\global\long\def\eps{\mathbb{\varepsilon}}%
\global\long\def\Z{\mathbb{Z}}%
\global\long\def\C{\mathbb{C}}%
\global\long\def\Qp{\mathbb{Q}_{p}}%
\global\long\def\Zp{\mathbb{Z}_{p}}%
\global\long\def\val{\mathbb{\mathrm{val}}}%
\global\long\def\Qp{\mathbb{Q}_{p}}%
\global\long\def\Zp{\mathbb{\mathbb{Z}}_{p}}%
\global\long\def\ac{\mathbb{\mathrm{ac}}}%
\global\long\def\S{\mathsection}%

\raggedbottom

\section{Introduction}

The goal of this paper is to explore a singularity invariant $\epsilon_{\star}(\phi)$
of a map $\phi$ between two manifolds over a local field. This invariant
quantifies the integrability of pushforward measures by $\phi$; we
define it in (\ref{eq:def-eps-star}) below after introducing some
notation that will also be used in the sequel.

Throughout this paper, we fix a local field $F$ of characteristic
$0$, i.e., $\mathbb{R}$, $\mathbb{C}$ or a finite extension of
$\mathbb{Q}_{p}$. If $X$ is an $F$-analytic manifold of dimension
$n$, let $(U_{\alpha}\subset X,\psi_{\alpha}:U_{\alpha}\to F^{n})_{\alpha\in\mathcal{A}}$
be an atlas. We denote by $C^{\infty}(X)$ the space of smooth functions
on $X$, i.e.\ functions $f:X\to\mathbb{C}$ such that $f\circ\psi_{\alpha}^{-1}|_{\psi_{\alpha}(U_{\alpha})}$
is smooth for each $\alpha\in\mathcal{A}$, and by $C_{c}^{\infty}(X)$
the subspace of compactly supported smooth functions (if $F$ is non-Archimedean,
smooth means locally constant). We similarly write $\mathcal{M}^{\infty}(X)$
for the space of smooth measures on $X$, i.e.\ measures such that
each $(\psi_{\alpha})_{*}(\mu|_{U_{\alpha}})$ has a smooth density
with respect to the Haar measure. We use $\mathcal{M}_{c}^{\infty}(X)$
to denote the space of smooth compactly supported measures on $X$.
For $1\leq q\leq\infty$, consider the class $\mathcal{M}_{c,q}(X)$
of finite Radon measures $\mu$ on $X$ that are compactly supported
and such that for any $\alpha\in\mathcal{A}$ the measure $(\psi_{\alpha})_{*}(\mu|_{U_{\alpha}})$
is absolutely continuous with density in $L^{q}(F^{n})$. All these
classes do not depend on the choice of the atlas. For $\mu\in\mathcal{M}_{c,1}(X)$
we define 
\[
\epsilon_{\star}(\mu):=\sup\left\{ \epsilon\geq0\,|\,\mu\in\mathcal{M}_{c,1+\epsilon}(X)\right\} .
\]
Note that by Jensen's inequality $\mu\in\mathcal{M}_{c,1+\epsilon}(X)$
for all $0\leq\epsilon<\epsilon_{\star}(\mu)$.

Now let $\phi:X\to Y$ be an analytic map between $F$-analytic manifolds
$X,Y$. If $\phi$ is \emph{locally dominant}, i.e.\ $\phi(U)$ contains
a non-empty open set for each open subset $U\subseteq X$, then $\phi_{*}\mu\in\mathcal{M}_{c,1}(Y)$
whenever $\mu\in\mathcal{M}_{c,1}(X)$. We now set for each $x\in X$,
\begin{equation}
\epsilon_{\star}(\phi;x):=\sup_{U\ni x}\inf_{\mu\in\mathcal{M}_{c}^{\infty}(U)}\epsilon_{\star}(\phi_{*}\mu)=\sup_{U\ni x}\inf_{\mu\in\mathcal{M}_{c,\infty}(U)}\epsilon_{\star}(\phi_{*}\mu),\label{eq:defeps-phi}
\end{equation}
where the supremum is over all open neighborhoods $U$ of $x$. Finally,
we can define 
\begin{equation}
\epsilon_{\star}(\phi):=\inf_{\mu\in\mathcal{M}_{c}^{\infty}(X)}\epsilon_{\star}(\phi_{*}\mu)=\inf_{x\in X}\epsilon_{\star}(\phi;x).\label{eq:def-eps-star}
\end{equation}
Note that if there exists $U\ni x$ such that $\phi_{*}\mu$ lies
in $\mathcal{M}_{c,q}(Y)$ for all $1<q<\infty$ and all $\mu\in\mathcal{M}_{c}^{\infty}(U)$,
then $\epsilon_{\star}(\phi;x)=\infty$. The best case scenario in
this setting is obtained when $\phi_{*}\mu\in\mathcal{M}_{c,\infty}(Y)$
for every $\mu\in\mathcal{M}_{c}^{\infty}(X)$. In this case we say
that $\phi$ is an\emph{ $L^{\infty}$-map.}

\medskip{}
The main motivation for $\epsilon_{\star}(\phi;x)$ comes from singularity
theory. In general, bad singularities of $\phi$ should manifest themselves
in poor analytic behavior of the pushforward $\phi_{*}\mu$ of $\mu\in\mathcal{M}_{c}^{\infty}(X)$.
This phenomenon has been extensively studied in the case $Y=F^{m}$,
through the analysis of the Fourier transform $\mathcal{F}(\phi_{*}\mu)$,
which takes the form of an \emph{oscillatory integral} in the Archimedean
case, and of an \emph{exponential sum} in the $p$-adic case. When
the dimension of the target space is equal to $m=1$, the rate of
decay of $\mathcal{F}(\varphi_{*}\mu)$ is closely related to singularity
invariants such as the log-canonical threshold (see e.g. \cite{Igu78},
\cite[Corollary 1.4.5]{Den91a} and \cite[Theorem 1.5]{CMN19} for
the non-Archimedean case, and \cite[Chapter 7]{AGV88} for the Archimedean
case). In higher dimension, the connections are less explicit, however,
milder singularities of $\phi$ still result in faster decay rates
of $\mathcal{F}(\phi_{*}\mu)$. For more information on $\mathcal{F}(\phi_{*}\mu)$,
we refer to Igusa's work \cite{Igu78}, the surveys of Denef, Meuser
and Le\'{o}n-Cardenal \cite{Den91a,LC22,Meu16}, as well as the book
\cite[Parts II, III]{AGV88} by Arnold, Guse\u{\i}n-Zade, and Varchenko,
and also the discussion in $\S$\ref{sub:Further discussion} below.

The study of the integrability properties of $\phi_{*}\mu$ and their
relation to the singularities of $\phi$ has received less treatment.
The invariant $\epsilon_{\star}(\phi;x)$ is a natural step in this
direction, and is more robust than Fourier-type invariants as it is
also meaningful when $Y$ is any smooth manifold, which is especially
important for applications (see $\mathsection$\ref{subsec:Application:-regularization-by}).

The invariant $\epsilon_{\star}(\phi;x)$ tends to be small whenever
the singularities of $\phi$ are bad near $x$. When $\phi$ is a
submersion, the pushforward $\phi_{*}\mu$ of any $\mu\in\mathcal{M}_{c}^{\infty}(X)$
is smooth and in particular lies in $\mathcal{M}_{c,\infty}(Y)$.
Moreover, Aizenbud and Avni \cite{AA16} have shown that for algebraic
maps $\varphi:X\rightarrow Y$ between smooth algebraic $\Q$-varieties,
the condition that the corresponding map $\varphi_{\Q_{p}}:X(\Q_{p})\rightarrow Y(\Q_{p})$
of $\Q_{p}$-analytic varieties is an $L^{\infty}$-map is equivalent
to a certain mild singularity property, namely that $\varphi$ is
flat with fibers of rational singularities (abbreviated (FRS), see
Definition~\ref{def:(FRS)}).

When analyzing $\epsilon_{\star}(\phi;x)$, one can further restrict
the infima in (\ref{eq:defeps-phi})\textendash (\ref{eq:def-eps-star})
to the class of compactly supported measures $\mu$ which are constructible,
in the sense of \cite[Section 3]{CGH14b} and \cite[Definition 1.2]{CM11}
(see also \cite{LR97}). This class is preserved under pushforward
by analytic maps, and therefore $\phi_{*}\mu$ is constructible as
well. Moreover, constructible measures admit a well behaved structure
theory and have tame analytic behavior (see e.g. \cite{CL08,CL10,CM11,CM13,CGH14b,CGH18}),
and, in particular, it follows from \cite{GH21,CM13} that $\epsilon_{\star}(\phi_{*}\mu)>0$.
Positivity of $\epsilon_{\star}(\phi_{*}\mu)$ in the real case can
further be deduced from \cite[Section 2]{RS88}. 

Our goal in this paper is to explore in more detail the properties
of $\epsilon_{\star}(\phi;x)$, and in particular to obtain upper
and lower bounds on $\epsilon_{\star}(\phi;x)$ in terms of other
singularity invariants which may be easier to compute, such as the
log-canonical threshold of certain ideals. In Theorem~\ref{thm:lowerbd}
(proved in $\mathsection$\ref{sec:lower bound}), we give a lower
bound on $\epsilon_{\star}(\phi;x)$. In particular, this provides
a proof for the positivity of $\epsilon_{\star}$, without using the
theory of motivic integration. In Theorem~\ref{thm:upper-C} (proved
in $\mathsection$\ref{sec:upper bound}) an upper bound on $\epsilon_{\star}(\phi;x)$
is given in the complex case. In Theorem~\ref{thm:one dimensional}
(proved in $\mathsection$\ref{sec:Formula-for-}) an explicit formula
is given for $\epsilon_{\star}(\phi;x)$ when $Y$ is one-dimensional,
over any local field. Finally, in Theorem \ref{thm:characterization of Lq for all q-intro}
(proved in $\mathsection$\ref{sec:Geometric-characterization-of}),
we specialize to polynomial maps between smooth algebraic varieties,
and geometrically characterize the condition $\epsilon_{\star}(\phi;x)=\infty$.

\subsection{\label{subsec:Application:-regularization-by}Application: regularization
by convolution}

Apart from the geometric motivation discussed above, an additional
source of motivation comes from the study of random walks on groups.
Assume that $G$ is an $F$-analytic group, and take a finite measure
$\nu$ on $G$. Can one find a number $k\in\mathbb{N}$ such that
the $k$-th convolution power $\nu^{*k}$ lies in $\mathcal{M}_{c,\infty}(G)$,
and if so, what is the smallest such number $k_{\star}(\nu)$?

An important class of examples comes from the realm of word maps.
Given a word $w$ in $r$ letters, by which we mean an element either
of the free group $F_{r}$, or of the free Lie algebra $\mathcal{L}_{r}$,
one can consider the corresponding word maps $w_{G}:G^{r}\rightarrow G$
or $w_{\mathfrak{g}}:\mathfrak{g}^{r}\rightarrow\mathfrak{g}$, where
$\mathfrak{g}$ is the Lie algebra of $G$. When $G$ is a compact
real or $p$-adic Lie group, $w$ induces a natural measure $\nu:=(w_{G})_{*}\pi_{G}$
where $\pi_{G}$ is the normalized Haar measure on $G$. Here, one
may further ask what is the $L^{\infty}$-mixing time of the word
measure $\nu$, namely, how large should $k$ be to ensure that $\left\Vert \nu^{*k}-\pi_{G}\right\Vert _{\infty}\ll1$.
Questions of this kind have been studied e.g. in \cite{AA16,LST19,GHb,AG,AGL}.

Similar problems have also appeared in a variety of other applications;
we mention the work of Ricci and Stein on singular integrals on non-abelian
groups (see the ICM survey of Stein \cite{Ste87}). In the study of
one-dimensional random operators, the regularity of the distribution
of transfer matrices (lying in SL($2,\mathbb{R}$), or, more generally,
in Sp($2W,\mathbb{R}$)) plays a role in the work of Shubin\textendash Vakilian\textendash Wolff
\cite{SVW98} as well as in the recent work \cite{GS22}. 

Motivated by the above examples, we focus on the following setting:
$G$ and $X$ are analytic, and $\nu=\phi_{*}\mu$ is the pushforward
of a measure $\mu\in\mathcal{M}_{c,\infty}(X)$ under a locally dominant
analytic map $\phi:X\to G$. For $x\in X$, let 
\[
k_{\star}(\phi;x)=\min_{U\ni x}\max_{\mu\in\mathcal{M}_{c}^{\infty}(U)}k_{\star}(\phi_{*}\mu),
\]
be the smallest $k$ that works for any $\mu$ supported in a sufficiently
small neighborhood $U$ of $x$.

According to Young's convolution inequality for locally compact groups
(see \cite[Corollary 2.3]{KR78} and Remark \ref{rem:Young for non-unimodular}),
\[
\nu\in\mathcal{M}_{c,1+\epsilon}(G)\Longrightarrow\nu^{*k}\in\mathcal{M}_{c,1+r}(G),\quad\text{where}\quad r=\begin{cases}
\frac{k\epsilon}{1-(k-1)\epsilon} & \text{if}\,k<\frac{1+\epsilon}{\epsilon},\\
\infty & \text{if}\,k\geq\frac{1+\epsilon}{\epsilon}.
\end{cases}
\]
This implies 
\begin{equation}
k_{\star}(\phi;x)\leq\left\lfloor \frac{1+\epsilon_{\star}(\phi;x)}{\epsilon_{\star}(\phi;x)}\right\rfloor +1<\infty.\label{eq:via-young}
\end{equation}
As mentioned above, a more classical approach to bounding $k_{\star}(\phi;x)$
relies on the study of the decay of the Fourier transform of $\phi_{*}\mu$.
While the Fourier-analytic approach often provides sharper bounds
(see \cite[Proposition 5.7]{GHb}), it is mainly applicable for abelian
groups such as $G=F^{n}$, or mildly non-abelian groups such as the
Heisenberg model. One can use non-commutative Fourier transform to
analyze compact Lie groups such as $\mathrm{SO}_{n}(\mathbb{R})$,
but such representation theoretic techniques are much less effective
for compact $p$-adic groups and non-compact Lie groups. To treat
the latter cases, one can use algebro-geometric techniques as in \cite{AA16,GHb};
however, this method requires some assumptions on $\phi$. Thus one
can argue that the approach to regularization via (\ref{eq:via-young})
is currently the most efficient one for treating $p$-adic analytic
groups and non-compact Lie groups such as $SL_{n}(\mathbb{R})$.

\subsection{\label{sub:main}Main results }

We now discuss the main results in this paper.

\subsubsection{\label{subsec:A-lower-bound}A lower bound on $\epsilon_{\star}$}

While the mere positivity of $\epsilon_{\star}$ (and the mere finiteness
of $k_{\star}$) are sufficient for some applications, other ones
require explicit bounds. Our first result provides a bound in terms
of an important exponent known as the log-canonical threshold, see
e.g. \cite{Mus12,Kol} for $F=\C$. For an analytic map $\psi:X\to F$,
define the \emph{$F$-log-canonical threshold} 
\begin{equation}
\operatorname{lct}_{F}(\psi;x):=\sup\left\{ s>0:\exists U\ni x\text{ s.t.}\,\forall\mu\in\mathcal{M}_{c,\infty}(U),\,\int_{X}\left|\psi(x)\right|_{F}^{-s}d\mu(x)<\infty\right\} ,\label{eq:def-lct}
\end{equation}
where $U$ runs over all open neighborhoods of $x$, and $\left|\,\cdot\,\right|_{F}$
is the absolute value on $F$, normalized such that $\mu_{F}(aA)=\left|a\right|_{F}\cdot\mu_{F}(A)$,
for all $a\in F^{\times}$, $A\subseteq F$, and where $\mu_{F}$
is a Haar measure on $F$. In particular, $\left|\cdot\right|_{\C}=\left|\cdot\right|^{2}$
is the square of the usual absolute value $\left|\cdot\right|$ on
$\C$. More generally, if $J$ is a non-zero ideal of analytic functions
generated by $\psi_{1},\cdots,\psi_{\ell}$, define
\begin{equation}
\operatorname{lct}_{F}(J;x):=\sup\left\{ s>0:\exists U\ni x\text{ \,s.t.\, }\forall\mu\in\mathcal{M}_{c,\infty}(U),\,\int_{X}\min_{1\leq i\leq l}\Big[\left|\psi_{i}(x)\right|_{F}^{-s}\Big]d\mu(x)<\infty\right\} .\label{eq:lct-analytic definition}
\end{equation}
This definition does not depend on the choice of the generators, and
thus it extends in a straightforward way to sheaves of ideals. Furthermore,
the log-canonical threshold is always strictly positive (see $\mathsection$\ref{sec:Embedded-resolution-of}).

Given a locally dominant analytic map $\phi:X\to Y$ between two $F$-analytic
manifolds, we define the \emph{Jacobian ideal sheaf} $\mathcal{J}_{\phi}$
as follows. If $X\subseteq F^{n}$ and $Y\subseteq F^{m}$ are open
subsets, we define $\mathcal{J}_{\phi}$ to be the ideal in the algebra
of analytic functions on $X$, generated by the $m\times m$-minors
of the differential $d_{x}(\phi)$ of $\phi$. Note that if $\psi_{1}:X\rightarrow X'\subseteq F^{n}$
and $\psi_{2}:Y\rightarrow Y'\subseteq F^{m}$ are analytic diffeomorphisms,
then $\psi_{1}^{*}\left(\mathcal{J}_{\psi_{2}\circ\phi\circ\psi_{1}^{-1}}\right)=\mathcal{J}_{\phi}$.
Hence, the definition of $\mathcal{J}_{\phi}$ can be generalized
(or glued) to an ideal sheaf on $X$, if $X$ and $Y$ are $F$-analytic
manifolds.

We now describe the first main result, which provides a lower bound
on $\epsilon_{\star}(\phi;x)$ in terms of the Jacobian ideal $\mathcal{J}_{\phi}$
of $\phi$. The proof is given in $\mathsection$\ref{sec:lower bound}. 
\begin{thm}
\label{thm:lowerbd}Let $X,Y$ be analytic $F$-manifolds, $\dim X=n\geq\dim Y=m$,
and let $\phi:X\to Y$ be a locally dominant analytic map. Then for
every $x\in X$, 
\begin{equation}
\epsilon_{\star}(\phi;x)\geq\operatorname{lct}_{F}(\mathcal{J}_{\phi};x);\label{eq:lowerbd}
\end{equation}
if $m=n$, equality is achieved. 
\end{thm}

As we will see later in Remark \ref{rem:discussion on upper bound}(2),
for $\dim X>\dim Y$ the inequality (\ref{eq:lowerbd}) may be strict.
In the next paragraph we discuss an additional case in which $\epsilon_{\star}$
can be computed explicitly: $\dim Y=1$.

\subsubsection{\label{subsec:A-formula-for}A formula in the one-dimensional case
and a reverse Young inequality}

When the target space $Y$ is one-dimensional, Hironaka's theorem
on the embedded resolution of singularities \cite{Hir64} provides
a powerful tool to study the structural properties of algebraic and
analytic maps. This theorem, as well the asymptotic expansion of pushforward
measures about a critical value of the map, allows us to obtain the
following much more detailed results, the proofs of which are given
in $\S$\ref{sec:Formula-for-}.

The first one is an exact formula relating $\epsilon_{\star}$ to
the log-canonical threshold. 
\begin{thm}
\label{thm:one dimensional}Let $X$ be an analytic $F$-manifold,
and let $\phi:X\to F$ be a locally dominant analytic map. Then for
each $x\in X$, we have: 
\begin{equation}
\epsilon_{\star}(\phi;x)=\begin{cases}
\infty & \text{if }\operatorname{lct}_{F}(\phi-\phi(x);x)\geq1,\\
\frac{\operatorname{lct}_{F}(\phi-\phi(x);x)}{1-\operatorname{lct}_{F}(\phi-\phi(x);x)} & \text{if }\operatorname{lct}_{F}(\phi-\phi(x);x)<1.
\end{cases}\label{eq:formula for epsilon one dimensional}
\end{equation}
\end{thm}

By Theorem \ref{thm:one dimensional}, by (\ref{eq:via-young}) and
by a Thom\textendash Sebastiani type result for $\operatorname{lct}_{F}$
(Proposition \ref{prop:Thom--Sebastiani}(1)), one can further show:
\begin{equation}
\left\lceil \frac{1}{\operatorname{lct}_{F}(\phi_{x};x)}\right\rceil \leq k_{\star}(\phi;x)\leq\left\lfloor \frac{1}{\operatorname{lct}_{F}(\phi_{x};x)}\right\rfloor +1.\label{eq:lct is equivalent to 1/k}
\end{equation}
We therefore see that that $\epsilon_{\star}(\phi;x),\operatorname{lct}_{F}(\phi_{x};x)$
and $1/k_{\star}(\phi;x)$ are asymptotically equivalent as $\operatorname{lct}_{F}(\phi_{x};x)\rightarrow0$.
In $\S$\ref{sub:Further discussion}, we shall see that these invariants
are also tightly related to an invariant $\delta_{\star}$ quantifying
Fourier decay. We will further see in $\S$\ref{subsec:Higher-dimensions}
that the close relation between all these quantities is a special
feature of the one-dimensional case, and does not generalize to higher
dimensions.\medskip{}

We next provide a reverse Young result for pushforward measures by
analytic maps. Recall that Young's convolution inequality (see e.g.\ \cite[pp.\ 54-55]{Wei40})
implies that 
\begin{equation}
\frac{\epsilon_{1}}{1+\epsilon_{1}}+\frac{\epsilon_{2}}{1+\epsilon_{2}}\geq\frac{\epsilon}{1+\epsilon}\Longrightarrow\mathcal{M}_{c,1+\epsilon_{1}}(F)*\mathcal{M}_{c,1+\epsilon_{2}}(F)\subseteq\mathcal{M}_{c,1+\epsilon}(F).\label{eq:Young's inequality}
\end{equation}
Using the connection between $\epsilon_{\star}(\phi;x)$ and $k_{\star}(\phi;x)$
as well as the structure of pushforward measures, we show the following
converse to (\ref{eq:Young's inequality}): 
\begin{thm}[Reverse Young inequality]
\label{thm:reverse Young-intro}Let $\nu_{1},\nu_{2}\in\mathcal{M}_{c,1}(F)$
be pushforward measures of the form $\nu_{j}=(\phi_{j})_{*}\mu_{j}$,
where $\mu_{j}\in\mathcal{M}_{c}^{\infty}(F^{n_{j}})$ and $\phi_{j}:F^{n_{j}}\to F$
are analytic, locally dominant. If $\nu_{1}*\nu_{2}\in\mathcal{M}_{c,1+\epsilon}(F)$
for some $\epsilon>0$, then 
\begin{equation}
\frac{\epsilon_{\star}(\nu_{1})}{1+\epsilon_{\star}(\nu_{1})}+\frac{\epsilon_{\star}(\nu_{2})}{1+\epsilon_{\star}(\nu_{2})}>\frac{\epsilon}{1+\epsilon}.\label{eq:revyoung}
\end{equation}
In particular, if $\nu_{1}$ is equal to $\nu_{2}$, it lies in $\mathcal{M}_{c,1+\frac{\epsilon}{2+\epsilon}}(F)$. 
\end{thm}

\begin{rem}
Under the assumptions of Theorem \ref{thm:reverse Young-intro}, if
$\nu_{1}*\nu_{2}\in\mathcal{M}_{c,\infty}(F)$ then (\ref{eq:revyoung})
(applied with $\epsilon\to+\infty$) implies 
\begin{equation}
\frac{\epsilon_{\star}(\nu_{1})}{1+\epsilon_{\star}(\nu_{1})}+\frac{\epsilon_{\star}(\nu_{2})}{1+\epsilon_{\star}(\nu_{2})}\geq1.\label{eq:Young for infty}
\end{equation}
In general, one cannot hope for a strict inequality in (\ref{eq:Young for infty});
indeed taking $F=\R$ or $F=\Qp$ for $p\equiv3\mod4$, $\phi_{1}=\phi_{2}=x^{2}$
and $\mu_{1}=\mu_{2}$ a uniform measure on some ball around $0$,
one has $\epsilon_{\star}(\nu_{1})=\epsilon_{\star}(\nu_{2})=1$,
so that (\ref{eq:Young for infty})~$=1$. On the other hand $\nu_{1}*\nu_{2}\in\mathcal{M}_{c,\infty}(F)$.

When $F=\C$, or more generally when the codimension of $(\phi_{1}*\phi_{2})^{-1}(0)$
in $F^{n_{1}+n_{2}}$ is $1$, we expect (\ref{eq:Young for infty})
to hold with a strict inequality. 
\end{rem}

\subsubsection{\label{subsec:Upper-bounds-on}Upper bounds on $\epsilon_{\star}$}

For $n>m$, the lower bound (\ref{eq:lowerbd}) may in general not
be an equality. However, the next result shows that when $F=\C$,
(\ref{eq:lowerbd}) is asymptotically sharp as $\operatorname{lct}_{\mathbb{C}}(\mathcal{J}_{\phi};x)\rightarrow0$. 
\begin{thm}
\label{thm:upper-C}Let $X,Y$ be analytic $\C$-manifolds, and let
$\phi:X\to Y$ be a locally dominant analytic map. Then, whenever
$\operatorname{lct}_{\C}(\mathcal{J}_{\phi};x)<1$, 
\begin{equation}
\epsilon_{\star}(\phi;x)\leq\frac{\operatorname{lct}_{\mathbb{C}}(\mathcal{J}_{\phi};x)}{1-\operatorname{lct}_{\mathbb{C}}(\mathcal{J}_{\phi};x)}.\label{eq:upper bound on epsilon}
\end{equation}
\end{thm}

\begin{rem}
\label{rem:discussion on upper bound}\, 
\begin{enumerate}
\item When $Y=F$, we have $\operatorname{lct}_{F}(\phi-\phi(x);x)\leq\operatorname{lct}_{F}(\mathcal{J}_{\phi};x)$
(see Proposition \ref{prop:Lbounds agree with formula}), so that
(\ref{eq:upper bound on epsilon}) holds also for $F\neq\C$, whenever
$\operatorname{lct}_{F}(\mathcal{J}_{\phi};x)<1$. In particular,
Theorem \ref{thm:one dimensional} implies Theorem \ref{thm:upper-C}
in the case that $Y=\C$. 
\item The upper bound (\ref{eq:upper bound on epsilon}) is asymptotically
tight, in the sense that the value of $\epsilon_{\star}$ can be arbitrarily
close to the upper bound (\ref{eq:upper bound on epsilon}), as seen
from the following family of examples. Let $\phi:=x_{1}^{m}x_{2}^{m}...x_{n}^{m}$.
Then $\nabla\phi=\langle x_{1}^{m-1}x_{2}^{m}...x_{n}^{m},...,x_{1}^{m}x_{2}^{m}...x_{n}^{m-1}\rangle$,
and thus by \cite[Main Theorem and Example 5]{How01}, it follows
that 
\[
\operatorname{lct}_{\C}(\mathcal{J}_{\phi};0)=\frac{1}{m-\frac{1}{n}},
\]
so that the upper bound in (\ref{eq:upper bound on epsilon}) becomes
$\epsilon_{\star}(\phi;0)\leq\frac{1}{m-1-\frac{1}{n}}$, whereas
the lower bound (\ref{eq:lowerbd}) is $\epsilon_{\star}(\phi;0)\geq\frac{1}{m-\frac{1}{n}}$.
We see that the true value $\epsilon_{\star}(\phi;0)=\frac{1}{m-1}$
is closer to the upper bound than to the lower bound. 
\end{enumerate}
\end{rem}

One may wonder whether Theorem \ref{thm:upper-C} can be extended
to $F\neq\C$. In the current proof, the volumes of balls in complex
manifolds are bounded from below using Lelong's monotonicity theorem,
and the latter fails for $F=\R$ and for any $\Q_{p}$. If $\varphi:X\to Y$
is a polynomial map between smooth varieties, defined over $\Q$,
we expect the upper bound in Theorem \ref{thm:upper-C} to hold for
$\varphi_{\Q_{p}}:X(\Q_{p})\to Y(\Q_{p})$ for infinitely many primes
$p$. This would follow from a positive answer to Question \ref{que:Archimedean to non-Archimedean}.

\subsubsection{Applications to convolutions of algebraic morphisms}

Throughout this and the next subsections we assume $K$ is a number
field. In \cite{GH19,GH21} and \cite{GHb}, the first two authors
have studied the following convolution operation in algebraic geometry: 
\begin{defn}
\label{def:convolution}Let $\varphi:X\rightarrow G$ and $\psi:Y\rightarrow G$
be morphisms from algebraic $K$-varieties $X,Y$ to an algebraic
$K$-group $G$. We define their \emph{convolution} by 
\[
\varphi*\psi:X\times Y\rightarrow G,\text{ }(x,y)\mapsto\varphi(x)\cdot_{G}\psi(y).
\]
We denote by $\varphi^{*k}:X^{k}\rightarrow G$ the \emph{$k$-th
self convolution} of $\varphi$. 
\end{defn}

We restrict ourselves to the setting where $X,Y$ are smooth algebraic
$K$-varieties and $G$ is a connected algebraic $K$-group. The main
motto is that the algebraic convolution operation has a smoothing
effect on morphisms, similarly to the usual convolution operation
in analysis (see \cite[Proposition 1.3]{GH19} and \cite[Proposition 3.1]{GH21}).
For example, starting from any dominant map $\varphi:X\rightarrow G$,
the $k$-th self convolution $\varphi^{*k}:X^{k}\rightarrow G$ is
flat for every $k\geq\dim G$ (\cite[Theorem B]{GH21}). To explain
the connection to this work, we introduce the following property: 
\begin{defn}[{{{{\cite[Definition II]{AA16}}}}}]
\label{def:(FRS)}~ 
\begin{enumerate}
\item A $K$-variety $Z$ has \textit{rational singularities} if it is normal
and for every resolution of singularities $\pi:\widetilde{Z}\rightarrow Z$,
the pushforward $\pi_{*}(\mathcal{O}_{\widetilde{Z}})$ of the structure
sheaf has no higher cohomologies. 
\item A morphism $\varphi:X\rightarrow Y$ between smooth $K$-varieties
is called \textit{(FRS)} if it is flat and if every fiber of $\varphi$
has rational singularities. 
\end{enumerate}
In \cite[Theorem 3.4]{AA16} (see Theorem \ref{thm:analytic criterion of the (FRS) property}
below), Aizenbud and Avni proved the following. A morphism $\varphi:X\rightarrow Y$
between smooth $K$-varieties, is (FRS) if and only if for every non-Archimedean
local field $F\supseteq K$, one has $\left(\varphi_{F}\right)_{*}\mu\in\mathcal{M}_{c,\infty}(Y(F))$
for every $\mu\in\mathcal{M}_{c}^{\infty}(X(F))$. A similar characterization
can be given for $F=\C$, see Corollary \ref{cor:characterization of Linfty}.

This characterization of $L^{\infty}$-maps allows one to study random
walks on analytic groups as in $\mathsection$\ref{subsec:Application:-regularization-by}
in an algebro-geometric way, via the above algebraic convolution operation.
Starting from a pushforward $\nu=\left(\varphi_{F}\right)_{*}\mu$
of $\mu\in\mathcal{M}_{c,\infty}(X(F))$ by an algebraic map $\varphi:X\to G$,
instead of showing that $\nu^{*k}\in\mathcal{M}_{c,\infty}(G(F))$,
it is enough to show that $\varphi^{*k}:X^{k}\rightarrow G$ is an
(FRS) morphism. This method was used in \cite{AA16,GHb} to study
word maps. Moreover, in \cite{GH19,GH21} it was shown that any locally
dominant morphism $\varphi:X\rightarrow G$ becomes (FRS) after sufficiently
many self-convolutions.

Using (\ref{eq:via-young}), Theorem \ref{thm:lowerbd} and Corollary
\ref{cor:characterization of Linfty}, explicit bounds can be given
on the required number of self convolutions, in terms of the Jacobian
ideal of $\varphi$. 
\end{defn}

\begin{cor}
\label{cor:applications for convolution}Let $X$ be a smooth $K$-variety,
$G$ be a connected $K$-algebraic group and let $\varphi:X\to G$
be a locally dominant morphism. Then $\varphi^{*k}$ is (FRS) for
any $k\geq\left\lfloor \frac{1+\operatorname{lct}_{\C}(\mathcal{J}_{\varphi})}{\operatorname{lct}_{\C}(\mathcal{J}_{\varphi})}\right\rfloor +1$. 
\end{cor}

\begin{rem}
\label{rem:Young for non-unimodular}In the setting of locally compact
groups, Young's convolution inequality is commonly stated under the
assumption that the group $G$ is unimodular. In \cite[Lemma 2.1 and Corollary 2.3]{KR78}
a version for non-unimodular groups is given; if $G$ is a locally
compact group, with modular character $\triangle:G\rightarrow\R_{>0}$,
and if $1\leq p,q,r\leq\infty$ satisfy $\frac{1}{p}+\frac{1}{q}=1+\frac{1}{r}$,
then we have $\left\Vert \mu*(\triangle^{1-\frac{1}{p}}\nu')\right\Vert _{r}\leq\left\Vert \mu\right\Vert _{p}\left\Vert \nu'\right\Vert _{q}$
whenever $\mu\in\mathcal{M}_{c,p}(G)$ and $\nu'\in\mathcal{M}_{c,q}(G)$.
However, since the modular character $\triangle:G\rightarrow\R_{>0}$
is a continuous homomorphism, it bounded on the compact support of
$\nu'$. Hence, for every $\mu\in\mathcal{M}_{c,p}(G)$ and $\nu\in\mathcal{M}_{c,q}(G)$,
we deduce that
\[
\left\Vert \mu*\nu\right\Vert _{r}=\left\Vert \mu*\left(\triangle^{1-\frac{1}{p}}(\triangle^{\frac{1}{p}-1}\cdot\nu)\right)\right\Vert _{r}\leq\left\Vert \mu\right\Vert _{p}\cdot\left\Vert \triangle^{\frac{1}{p}-1}\cdot\nu\right\Vert _{q}\leq C\cdot\left\Vert \mu\right\Vert _{p}\cdot\left\Vert \nu\right\Vert _{q},
\]
for some constant $C$ depending on $G$ and $\nu$. In particular,
$\mu*\nu\in\mathcal{M}_{c,r}(G)$. 
\end{rem}

\subsubsection{An algebraic characterization of $\epsilon_{\star}=\infty$}

Let $\varphi:X\rightarrow Y$ be a morphism between smooth $K$-varieties.
We would like to characterize the condition that $\epsilon_{\star}(\varphi_{F})=\infty$
for all $F$ in certain families of local fields, in terms of the
singularities of $\varphi$. The singularity properties we consider
play a central role in birational geometry (see \cite{Kol97}).

Let $X$ be a normal $K$-variety, and let $\omega\in\Omega^{\mathrm{top}}(X_{\mathrm{sm}})$
be a rational top form on the smooth locus $X_{\mathrm{sm}}$ of $X$.
The zeros and poles of $\omega$ give rise to a divisor $\operatorname{div}(\omega)$
on $X$. Let $\pi:\widetilde{X}\rightarrow X$ be a resolution of
singularities, namely, a proper morphism from a smooth variety $\widetilde{X}$,
which is an isomorphism over $X_{\mathrm{sm}}$. Then $\pi^{*}\omega$
defines a unique rational top form on $\widetilde{X}$. Moreover,
when $X$ is nice enough (e.g. if $X$ is a local complete intersection),
$\operatorname{div}(\omega)$ is $\Q$-Cartier, and we can define
its pullback $\pi^{*}\operatorname{div}(\omega)$. The divisor $K_{\widetilde{X}/X}:=\operatorname{div}(\pi^{*}\omega)-\pi^{*}\operatorname{div}(\omega)$
on $\widetilde{X}$ is called the \emph{relative canonical divisor},
and one can verify that it does not depend on the choice of $\omega$.
$K_{\widetilde{X}/X}$ can be written as $K_{\widetilde{X}/X}=\sum_{i=1}^{M}a_{i}E_{i}$,
for some prime divisors $E_{i}$, $a_{i}\in\Q$. We say that $X$
has \emph{canonical singularities} (resp. \emph{log-canonical singularities}),
if $a_{i}\geq0$ (resp.\ $a_{i}\geq-1$) for all $1\leq i\leq M$.
When $X$ is a local complete intersection (e.g.\ a fiber of a flat
morphism between smooth schemes), canonical singularities are equivalent
to rational singularities. Let us give an example: 
\begin{example}
\label{exa:canonical sing}~Let $X\subseteq\mathbb{A}_{\C}^{n}$
be the variety defined by $\sum_{i=1}^{n}x_{i}^{d_{i}}=0$, with $n\geq3$.
Then $X$ has canonical singularities if and only if $\sum_{i=1}^{n}\frac{1}{d_{i}}>1$,
and log-canonical singularities if and only if $\sum_{i=1}^{n}\frac{1}{d_{i}}\geq1$. 
\end{example}

As seen from Example \ref{exa:canonical sing}, log-canonical singularities
are very close to being canonical, so one could suspect being flat
with fibers of log-canonical singularities is equivalent to $\epsilon_{\star}(\varphi_{\Q_{p}})=\infty$,
that is to being almost in $L^{\infty}$. Unfortunately, the normality
hypothesis required for log-canonical singularities turns out to be
too strong. For example, the map $\varphi(x,y)=xy$ satisfies $\epsilon_{\star}(\varphi_{\Q_{p}})=\infty$
for all $p$, but the fiber over $0$ is not normal (see $\mathsection$\ref{sec:Geometric-characterization-of},
after Corollary \ref{cor:characterization of Linfty}). This technical
issue can be resolved by considering the slightly weaker notion of
\emph{semi-log-canonical singularities}, which is an analogue of log-canonical
singularities for demi-normal schemes (see \cite[Section 4]{KSB88}).
Indeed, the variety $\left\{ xy=0\right\} $ is demi-normal and has
semi-log-canonical singularities. We can now state the main result
of this section: 
\begin{thm}
\label{thm:characterization of Lq for all q-intro}Let $\varphi:X\rightarrow Y$
be a map between smooth $K$-varieties. Then the following are equivalent: 
\begin{enumerate}
\item $\varphi$ is flat with fibers of semi-log-canonical singularities. 
\item For every local field $F$ containing $K$, we have $\epsilon_{\star}(\varphi_{F})=\infty$,
that is, for every $\mu\in\mathcal{M}_{c}^{\infty}(X(F))$, the measure
$\varphi_{*}\mu$ lies in $\mathcal{M}_{c,q}(Y(F))$ for all $1<q<\infty$.
\item For every large enough prime $p$, such that $\Qp\supseteq K$, we
have $\epsilon_{\star}(\varphi_{\Q_{p}})=\infty$. 
\item We have $\epsilon_{\star}(\varphi_{\C})=\infty$. 
\end{enumerate}
\end{thm}

We prove Theorem \ref{thm:characterization of Lq for all q-intro}
by showing the implications $(1)\Rightarrow(2)\Rightarrow(3)\Rightarrow(1)$
and $(1)\Rightarrow(2)\Rightarrow(4)\Rightarrow(1)$. The implications
$(2)\Rightarrow(3)$ and $(2)\Rightarrow(4)$ are immediate. In the
proof of $(3)\Rightarrow(1)$ and $(4)\Rightarrow(1)$, we reduce
to the case $Y=\mathbb{A}^{n}$, and show that $\varphi:X\rightarrow Y$
satisfies that $\varphi*\psi$ is (FRS) for every dominant map $\psi:Y'\rightarrow Y$.
By analyzing the jets of $\varphi$, and using a jet-scheme interpretation
of semi-log canonical singularities (Lemma \ref{lemma: jet flat vs (FRS) and (FSLCS)}),
we deduce Item (1). The proof also uses the Archimedean counterpart
of \cite[Theorem 3.4]{AA16}, which is stated in Corollary \ref{cor:characterization of Linfty}.

In the proof of $(1)\Rightarrow(2)$, we reduce to showing Item (2)
for constructible measures, and utilize their structure theory, namely,
we use \cite{CM11,CM12} for the Archimedean case, and \cite{CGH18}
for the non-Archimedean case.

\subsection{\label{subsec:Future-directions-and}Future directions and further
applications }

\subsubsection{\label{subsec:invariant of singularities}$\epsilon_{\star}$ as
an invariant of singularities}

Similarly to the analytic definition of $\epsilon_{\star}(\phi)$
in (\ref{eq:def-eps-star}), one can also define an algebro-geometric
invariant. Let $\operatorname{Loc}_{0}$ be the collection of all
non-Archimedean local fields $F$ of characteristic $0$. 
\begin{defn}
\label{def:epsilon-algebro geometric}Let $\varphi:X\rightarrow Y$
be a morphism between smooth $\Q$-varieties. We define $\epsilon_{\star}(\varphi):=\min\limits _{F\in\operatorname{Loc}_{0}}\epsilon_{\star}(\varphi_{F})$,
and $\epsilon_{\star}(\varphi;x)=\sup\limits _{U\ni x}\,\epsilon_{\star}(\varphi|_{U})$,
where $U$ varies over all Zariski open neighborhoods of $x\in X(\overline{\Q})$. 
\end{defn}

It is a consequence of Theorem \ref{thm:lowerbd} that $\epsilon_{\star}(\varphi)>0$
for any $\varphi$, and it essentially follows from \cite{GH21} that
$\epsilon_{\star}(\varphi)\in\Q$. We further expect $\epsilon_{\star}(\varphi;x)$
to have a purely algebro-geometric formula, and to have a good behavior
in families, which means the following. Suppose that $\varphi:\widetilde{X}\rightarrow\widetilde{Y}$
is a morphism over $\mathbb{A}_{\C}^{1}$, where $\pi_{1}:\widetilde{X}\rightarrow\mathbb{A}_{\C}^{1}$
and $\pi_{2}:\widetilde{Y}\rightarrow\mathbb{A}_{\C}^{1}$ are smooth
morphisms. This gives a family $\left\{ \varphi_{t}:\widetilde{X}_{t}\rightarrow\widetilde{Y}_{t}\right\} _{t\in\mathbb{A}_{\C}^{1}}$
of morphisms between smooth varieties. It follows from \cite{Var83},
that the function $x\mapsto\operatorname{lct}(\varphi_{\pi_{1}(x)};x)$
is lower semicontinuous. By Theorem \ref{thm:one dimensional}, $x\mapsto\epsilon_{\star}(\varphi_{\pi_{1}(x)};x)$
is lower semicontinuous as well, if $\pi_{2}:\widetilde{Y}\rightarrow\mathbb{A}^{1}$
has fibers of dimension $1$. We expect $x\mapsto\epsilon_{\star}(\varphi_{\pi_{1}(x)};x)$
to be lower semicontinuous in general. We further expect the following
question to have a positive answer.
\begin{question}
\label{que:Archimedean to non-Archimedean}Let $\varphi:X\rightarrow Y$
be as in Definition \ref{def:epsilon-algebro geometric}. Is it true
that $\epsilon_{\star}(\varphi_{\C})=\epsilon_{\star}(\varphi)$? 
\end{question}

Given a $\Q$-morphism $\varphi:X\rightarrow Y$ as above, one may
further wonder whether the quantity $\epsilon_{\star}(\varphi_{\C};x)=\sup_{U\ni x}\inf_{\mu\in\mathcal{M}_{c,\infty}(U)}\epsilon_{\star}((\varphi_{\C})_{*}\mu)$
in (\ref{eq:defeps-phi}) stays the same when the supremum is taken
over all Zariski open neighborhoods $U$ of $x$ instead of analytic
ones. 

\subsubsection{\label{subsec:invariant of words}$\epsilon_{\star}$ as an invariant
of words}

As discussed in $\mathsection$\ref{subsec:Application:-regularization-by},
one particularly interesting potential application is to the study
of word map on semisimple algebraic groups. In \cite[Theorem A]{GHb}
it was shown that Lie algebra word maps $w_{\mathfrak{g}}:\mathfrak{g}^{r}\rightarrow\mathfrak{g}$,
where $\mathfrak{g}$ is a simple Lie algebra, become (FRS) after
$\sim\operatorname{deg}(w)^{4}$ self convolutions, where $\operatorname{deg}(w)$
is the degree of $w$.
\begin{question}
\label{que:conjecture on (FRS) and flatness}Can we find $\alpha,C>0$
such that for any $w\in F_{r}$ of length $\ell(w)$, and every simple
algebraic group $G$, the word map $w_{G}^{*C\ell(w)^{\alpha}}$ is
(FRS)? 
\end{question}

A potential way to tackle Question \ref{que:conjecture on (FRS) and flatness}
is by studying $\epsilon_{\star}(w_{G})$ in the sense of Definition
\ref{def:epsilon-algebro geometric}; For each $w\in F_{r}$ (resp.
$w\in\mathcal{L}_{r}$), we define $\epsilon_{\star}(w):=\inf_{G}\,\epsilon_{\star}(w_{G})$
(resp. $\epsilon_{\star}(w):=\inf_{\mathfrak{g}}\,\epsilon_{\star}(w_{\mathfrak{g}})$),
where $G$ runs over all simple, simply connected algebraic groups,
and $\mathfrak{g}=\operatorname{Lie}(G)$. We can now ask the following:
\begin{question}
\label{ques:epsilon of words}Can we find for each $l\in\N$, a constant
$\epsilon(l)>0$ such that: 
\begin{enumerate}
\item For every $w\in F_{r}$ of length $\ell(w)=l$, we have $\epsilon_{\star}(w)\geq\epsilon(l)$? 
\item For every $w\in\mathcal{L}_{r}$ of degree $\operatorname{deg}(w)=l$,
we have $\epsilon_{\star}(w)\geq\epsilon(l)$? 
\end{enumerate}
\end{question}

\subsubsection{\label{subsec:representations}$\text{\ensuremath{\epsilon_{\star}}}$
as an invariant of representations}

In a recent work \cite{GGH} of the first two authors with Julia Gordon,
we apply the results and point of view of this paper to the realm
of representation theory, and define and study a new invariant of
representations of reductive groups $G$ over local fields. Harish-Chandra's
regularity theorem says that every character $\Theta(\pi)$ of an
irreducible representation $\pi$ of $G$ is given by a locally $L^{1}$-function.
Since characters are of motivic nature, a variant of \cite[Theorem F]{GH21}
suggests that they should in fact be locally in $L^{1+\epsilon}$,
for some $\epsilon>0$. This gives rise to a new invariant $\epsilon_{\star}(\pi)$,
which is not equivalent to previously known invariants of representations,
such as the Gelfand\textendash Kirillov dimension (see e.g. \cite{Vog78}).
We use a geometric construction and Theorem \ref{thm:lowerbd} to
provide a formula for $\epsilon_{\star}(\pi)$ in terms of the nilpotent
orbits appearing the local character expansion of $\pi$ (see \cite[Theorems A and D]{GGH}).

\subsection{\label{subsec:Conventions}Conventions}
\begin{enumerate}
\item By $\N$ we mean the set $\{0,1,2,...\}$. 
\item We use $K$ (resp. $F$) to denote number (resp. local) fields and
$\mathcal{O}_{K}$ (resp. $\mathcal{O}_{F}$) for their rings of integers. 
\item Given an algebraic map $\varphi:X\rightarrow Y$ between smooth $K$-varieties,
and a local field $F$, we denote by $\varphi_{F}:X(F)\rightarrow Y(F)$
the corresponding $F$-analytic map. We sometimes write $\varphi$
instead of $\varphi_{F}$ if the local field $F$ is clear from the
setting. 
\item If $f,g:X\to\mathbb{R}$ are two functions, we write $f\lesssim g$
if there exists a number $C>0$, possibly depending on $X$, $f$
and $g$, such that $f\leq Cg$. We write $f\sim g$ if $g\lesssim f\lesssim g$. 
\item We write $d_{x}\phi$ for the differential of an analytic map $\phi:X\rightarrow Y$
at $x\in X$, and we denote by $\operatorname{Jac}_{x}\phi:=\det(d_{x}\phi)$
the Jacobian of $\phi$ if $\dim X=\dim Y$. The generalization of
the Jacobian to the case of unequal dimensions is defined in $\S$\ref{sub:main}. 
\item Throughout the paper, we write $\left|\,\cdot\,\right|_{F}$ for the
absolute value on $F$, normalized so that $\mu_{F}(aA)=\left|a\right|_{F}\cdot\mu_{F}(A)$,
for all $a\in F^{\times}$, $A\subseteq F$, and where $\mu_{F}$
is a Haar measure on $F$. Note that $\left|\,\cdot\,\right|_{\C}$
is the square of the usual absolute value on $\C$. 
\item We write $H_{n}$ for the $n$-dimensional Hausdorff measure. We recall
that given a metric space $X$ and a subset $Z\subseteq X$, we define
$H_{n}(Z):=\underset{\delta\rightarrow0}{\sup}H_{n}^{\delta}(Z)$,
where
\[
H_{n}^{\delta}(Z)=\inf\left\{ \sum_{i=1}^{\infty}\alpha(n)\left(\mathrm{diam}(U_{i})\right)^{n}:\bigcup_{i=1}^{\infty}U_{i}\supseteq Z,\mathrm{diam}(U_{i})<\delta\right\} ,
\]
where $\mathrm{diam}(A)$ denotes the diameter of a set $A\subseteq X$
and $\alpha(n):=\frac{(\pi/4)^{n/2}}{\Gamma(\frac{n}{2}+1)}$. The
normalization constant $\alpha(n)$ is chosen such that $H_{n}$ coincides
with the Lebesgue measure in the case of $X=\R^{n}$. 
\item All the measures we consider are non-negative, unless stated otherwise. 
\end{enumerate}
\begin{acknowledgement*}
We thank Nir Avni, Joseph Bernstein, Lev Buhovski, Raf Cluckers and
Stephan Snegirov for useful conversations and correspondences. We
thank David Kazhdan for suggesting that the condition of semi-log-canonical
singularities should play a role in Theorem \ref{thm:characterization of Lq for all q-intro}.
We are particularly grateful to Rami Aizenbud for numerous discussions
on this project, and for initiating the collaboration between the
authors.

SS was supported in part by the European Research Council starting
grant 639305 (SPECTRUM), a Royal Society Wolfson Research Merit Award
(WM170012), and a Philip Leverhulme Prize of the Leverhulme Trust
(PLP-2020-064). 
\end{acknowledgement*}

\section{\label{sec:Embedded-resolution-of}Preliminaries: embedded resolution
of singularities}

Let $F$ be a local field of characteristic zero. We use the following
analytic version of Hironaka's theorem \cite{Hir64} on embedded resolution
of singularities. The map $\pi:\widetilde{X}\rightarrow U$ below
is called a \emph{log-principalization }(or \emph{uniformization})
of $J$. 
\begin{thm}[{{{See \cite[Theorem 2.3]{VZG08}, \cite[Theorem 2.2]{DvdD88}, \cite{BM89}
and \cite{Wlo09}}}}]
\label{thm:analytic resolution of sing}Let $U\subseteq F^{n}$ be
an open subset, and let $f_{1},...,f_{r}:U\rightarrow F$ be $F$-analytic
maps, generating a non-zero ideal $J$ in the algebra of $F$-analytic
functions on $U$. Then there exist an $F$-analytic manifold $\widetilde{X}$,
a proper $F$-analytic map $\pi:\widetilde{X}\rightarrow U$ and a
collection of closed submanifolds $\left\{ E_{i}\right\} _{i\in T}$
of $\widetilde{X}$ of codimension $1$, equipped with pairs of non-negative
integers $\left\{ (a_{i},b_{i})\right\} _{i\in T}$, such that the
following hold: 
\begin{enumerate}
\item $\pi$ is locally a composition of a finite number of blow-ups at
closed submanifolds, and is an isomorphism over the complement of
the common zero set $\mathrm{V}(J)$ of $J$ in $U$. 
\item For every $c\in\widetilde{X}$, there are local coordinates $(\widetilde{x}_{1},...,\widetilde{x}_{n})$
in a neighborhood $V\ni c$, such that each $E_{i}$ containing $c$
is given by the equation $\widetilde{x}_{i}=0$. Moreover, if without
loss of generality $E_{1},...,E_{m}$ contain $c$, then there exists
an $F$-analytic unit $v:V\rightarrow F$, such that the pullback
of $J$ is the principal ideal 
\begin{equation}
\pi^{*}J=\langle\widetilde{x}_{1}^{a_{1}}\cdots\widetilde{x}_{m}^{a_{m}}\rangle\label{eq:prinipalization}
\end{equation}
and such that the Jacobian of $\pi$ (i.e. $\det\,d_{\widetilde{x}}\pi$)
is given by: 
\begin{equation}
\operatorname{Jac}_{\widetilde{x}}(\pi)=v(\widetilde{x})\cdot\widetilde{x}_{1}^{b_{1}}\cdots\widetilde{x}_{m}^{b_{m}}.\label{eq:principalization of Jacobian}
\end{equation}
\end{enumerate}
\end{thm}

\begin{rem}
\label{rem:more on Hironakas therem}~ 
\begin{enumerate}
\item Condition (\ref{eq:prinipalization}) means that for each $f_{i}:U\rightarrow F$,
one can write $f_{i}\circ\pi(\widetilde{x})=u_{i}(\widetilde{x})\widetilde{x}_{1}^{a_{1}}\cdots\widetilde{x}_{m}^{a_{m}}$
for some analytic functions $u_{i}$, and $u_{j}(0)\neq0$ for at
least one $j\in\{1,...,r\}$. Note that the $a_{1},...,a_{m}$ are
the same for all the $f_{i}$'s. 
\item If $r=1$, so that $J=\langle f\rangle$, we may further assume that
$f\circ\pi(\widetilde{x}_{1},...,\widetilde{x}_{n})=C\cdot\widetilde{x}_{1}^{a_{1}}\cdots\widetilde{x}_{m}^{a_{m}}$
locally on each chart, for some constant $C\neq0$. Indeed, $u(0)\neq0$.
If $u(0)$ is an $a_{1}$-th power in $F$ then the same holds for
$u(\widetilde{x})$ in a small neighborhood of $0$. In this case,
we may apply the change of coordinates 
\[
(\widetilde{x}_{1},...,\widetilde{x}_{n})\mapsto(u(\widetilde{x})^{-\frac{1}{a_{1}}}\widetilde{x}_{1},...,\widetilde{x}_{n}).
\]
If $u(0)$ is not an $a_{1}$-th power we may multiply it by $b\in F^{\times}$
such that $u(\widetilde{x})b$ is an $a_{1}$-th power, and apply
a similar change of coordinates. 
\end{enumerate}
\end{rem}

The next lemma follows directly by changing coordinates using a log-principalization
of $J$ and computing the integral $\int\min\limits _{1\leq i\leq l}\big[\left|f_{i}(x)\right|^{-s}\big]d\mu(s)$
with respect to the new coordinates. 
\begin{lem}[{See e.g. \cite[Theorem 1.2]{Mus12} and \cite[Theorem 2.7]{VZG08}}]
\label{Lemma:F-lct}Let $J=\langle f_{1},...,f_{r}\rangle$ be an
ideal of $F$-analytic functions on $U\subseteq F^{n}$. Let $\pi:\widetilde{X}\rightarrow U$
be a log-principalization of $J$, with data $\left\{ E_{i}\right\} _{i\in T}$
and $\left\{ (a_{i},b_{i})\right\} _{i\in T}$ as in Theorem \ref{thm:analytic resolution of sing}.
Then the $F$-log-canonical threshold of $J$ at $x\in U$ is equal
to: 
\begin{equation}
\operatorname{lct}_{F}(J;x)=\min_{i:\,x\in\pi(E_{i})}\frac{b_{i}+1}{a_{i}}.\label{eq:lct algebraic}
\end{equation}
\end{lem}

\begin{rem}
\label{rem:lct vs F-lct}Note that while the data $\left\{ E_{i}\right\} _{i\in T}$
and $\left\{ (a_{i},b_{i})\right\} _{i\in T}$ depends on the log-principalization
$\pi$, Definition~\ref{eq:lct-analytic definition} is intrinsic,
thus the right-hand side of (\ref{eq:lct algebraic}) does not depend
on $\pi$. 
\end{rem}

\section{\label{sec:Formula-for-}The one-dimensional case}

\subsection{A formula for $\epsilon_{\star}$}

In this section we provide a formula for $\epsilon_{\star}$ in the
one-dimensional case (Theorem \ref{thm:one dimensional}). The formula
will be phrased in terms of the $F$-log-canonical threshold, where
$F$ is any local field of characteristic zero.

When $F$ is non-Archimedean, we denote by $\mathcal{O}_{F}$ its
ring of integers, by $k_{F}$ its residue field, and by $q_{F}$ the
number of elements in $k_{F}$. Write $\varpi_{F}\in\mathcal{O}_{F}$
for a fixed uniformizer (i.e., a generator of the maximal ideal of
the ring of integers) of $F$, and let $\left|\,\cdot\,\right|_{F}$
be as in $\mathsection$\ref{subsec:Conventions}, so that $\left|\varpi_{F}\right|_{F}=q_{F}^{-1}$.
We write $\mu_{F}$ for the Haar measure on $F$, normalized such
that $\mu_{F}(\mathcal{O}_{F})=1$ when $F$ is non-Archimedean, and
such that $\mu_{F}$ is the Lebesgue measure when $F$ is Archimedean.
We write $\mu_{\mathcal{O}_{F}}:=\mu_{F}|_{\mathcal{O}_{F}}$. We
write $dx$ instead of $\mu_{F}$ when we integrate a function $g(x)$
with respect to $\mu_{F}$. We denote by $\left\Vert (a_{1},...,a_{n})\right\Vert _{F}:=\max_{i}\left|a_{i}\right|_{F}$
the maximum norm on $F^{n}$.

For an analytic map $\phi:X\rightarrow F$, and $x\in X$, we set
$\phi_{x}(z):=\phi(z)-\phi(x)$.

To prove Theorem \ref{thm:one dimensional}, we reduce to the monomial
case using Hironaka's resolution of singularities ($\mathsection$\ref{sec:Embedded-resolution-of}),
and prove the monomial case in Lemma \ref{lem:formula for monomial maps}.
However, we first note that the upper bound in (\ref{eq:formula for epsilon one dimensional})
can be proved by elementary arguments, as follows: 
\begin{lem}
\label{lem:upper bound epsilon n=00003D00003D1}Let $X$ be an analytic
$F$-manifold, and let $\phi:X\to F$ be a locally dominant analytic
map. Then for every $x\in X$ with $\operatorname{lct}_{F}(\phi_{x};x)<1$,
one has: 
\[
\epsilon_{\star}(\phi;x)\leq\frac{\operatorname{lct}_{F}(\phi_{x};x)}{1-\operatorname{lct}_{F}(\phi_{x};x)}.
\]
\end{lem}

\begin{proof}
We need to show 
\[
1-\frac{1}{1+\epsilon_{\star}(\phi;x)}\leq\operatorname{lct}_{F}(\phi_{x};x).
\]
Let $\epsilon<\epsilon_{\star}(\phi;x)$. Then there exists a neighborhood
$U$ of $x$ such that $\phi_{*}\sigma\in L^{1+\epsilon}$ for every
$\sigma\in\mathcal{M}_{c,\infty}(U)$. Write $\phi_{*}\sigma=g(t)\mu_{F}$.
Let $B(a,\delta):=\left\{ t:\left|t-a\right|_{F}\leq\delta\right\} $,
and note that $\mu_{F}(B(a,\delta))\sim\delta$. By Jensen's inequality
we have: 
\[
(\phi_{*}\sigma)(B(a,\delta))=\frac{1}{\delta}\int_{B(a,\delta)}\delta g(t)dt\leq\left(\frac{1}{\delta}\int_{B(a,\delta)}\left(\delta g(t)\right)^{1+\epsilon}dt\right)^{\frac{1}{1+\epsilon}}\lesssim\delta^{1-\frac{1}{1+\epsilon}},
\]
i.e.\ we have the distributional estimate $\sigma\left(\left\{ z:\left|\phi(z)-\phi(x)\right|_{F}\leq\delta\right\} \right)\lesssim\delta^{1-\frac{1}{1+\epsilon}}$.
Using Fubini's theorem, we obtain: 
\begin{align*}
\int_{U}\left|\phi_{x}(z)\right|_{F}^{-s}d\sigma & =\int_{U}\left(\int_{0}^{\infty}1_{\left\{ (t,z)\,:\left|\phi_{x}(z)\right|_{F}^{-s}\geq t\right\} }(t,z)dt\right)d\sigma\\
 & =\int_{0}^{\infty}\sigma\left\{ z:\left|\phi(z)-\phi(x)\right|_{F}^{-s}\geq t\right\} dt\\
 & =\int_{0}^{\infty}\sigma\left\{ z:\left|\phi(z)-\phi(x)\right|_{F}\leq t^{-1/s}\right\} dt\\
 & \lesssim1+\int_{1}^{\infty}t^{-\frac{1}{s}\left(1-\frac{1}{1+\epsilon}\right)}dt<\infty,
\end{align*}
whenever $s<1-\frac{1}{1+\epsilon}$. This implies that $\operatorname{lct}_{F}(\phi_{x};x)\geq1-\frac{1}{1+\epsilon_{\star}(\phi;x)}$. 
\end{proof}
We now return to the main narrative. Specializing to the setting of
Hironaka's theorem, we consider the monomial case. 
\begin{lem}
\label{lem:formula for monomial maps}Let $f:F^{n}\rightarrow F$
be a monomial map $f(x_{1},...,x_{n})=x_{1}^{a_{1}}\cdot...\cdot x_{n}^{a_{n}}$,
let $g:F^{n}\rightarrow\R_{\geq0}$ be a continuous function and let
$\mu=g(x)\left|x_{1}\right|_{F}^{b_{1}}\cdot...\cdot\left|x_{n}\right|_{F}^{b_{n}}\mu_{F}^{n}$,
for $a_{1},...,a_{n}\in\Z_{\geq1}$ and $b_{1},...,b_{n}\in\Z_{\geq0}$.
Then: 
\[
\epsilon_{\star}(f_{*}\mu)\,\begin{cases}
=\infty & \text{if\, }\min_{i}\frac{b_{i}+1}{a_{i}}\geq1,\\
\geq\frac{\min_{i}\frac{b_{i}+1}{a_{i}}}{1-\min_{i}\frac{b_{i}+1}{a_{i}}} & \text{otherwise}.
\end{cases}
\]
Furthermore, if $g(0)\neq0$, the second bound is in fact an equality. 
\end{lem}

\begin{proof}
Without loss of generality we may assume $\min_{i}\frac{b_{i}+1}{a_{i}}=\frac{b_{n}+1}{a_{n}}$.
We first consider the special case when $\frac{b_{n}+1}{a_{n}}$ is
the unique minimum. We write $f$ as a composition $f=\operatorname{pr}\circ\psi$,
where $\psi:F^{n}\rightarrow F^{n}$ is given by $\psi(x_{1},...,x_{n})=(x_{1},...,x_{n-1},x_{1}^{a_{1}}\cdot...\cdot x_{n}^{a_{n}})$
and $\operatorname{pr}$ is the projection to the last coordinate.
Write $x=(\overline{x},x_{n})$, where $\overline{x}=(x_{1},...,x_{n-1})$,
and similarly $y=(\overline{y},y_{n})$. Note that $\operatorname{Jac}_{(x_{1},...,x_{n})}(\psi)=a_{n}x_{1}^{a_{1}}\cdot...\cdot x_{n-1}^{a_{n-1}}x_{n}^{a_{n}-1}$
and that if $\psi(\overline{x},x_{n})=(\overline{y},y_{n})$, then
$\overline{x}=\overline{y}$ and $x_{n}^{a_{n}}=y_{n}y_{1}^{-a_{1}}\cdot...\cdot y_{n-1}^{-a_{n-1}}$.
Hence the Radon\textendash Nikodym density of $\psi_{*}\mu$ is equal
to 
\begin{align*}
\frac{d(\psi_{*}\mu)}{d(\mu_{F}^{n})}(\overline{y},y_{n}) & =\sum_{t:\psi(\overline{y},t)=y}\frac{\left(\prod_{j=1}^{n-1}\left|y_{j}\right|_{F}^{b_{j}}\right)\left|t\right|_{F}^{b_{n}}g(\overline{y},t)}{\left|\operatorname{Jac}_{(y_{1},...,y_{n-1},t)}(\psi)\right|_{F}}\\
 & =\sum_{t:\psi(\overline{y},t)=y}g(\overline{y},t)\left|a_{n}\right|_{F}^{-1}\cdot\left(\prod_{j=1}^{n-1}\left|y_{j}\right|_{F}^{b_{j}-a_{j}}\right)\left|t\right|_{F}^{b_{n}-a_{n}+1}\\
 & =\left(\sum_{t:\psi(\overline{y},t)=y}g(\overline{y},t)\right)\cdot\left|a_{n}\right|_{F}^{-1}\left(\prod_{j=1}^{n-1}\left|y_{j}\right|_{F}^{b_{j}-a_{j}-\frac{(b_{n}-a_{n}+1)a_{j}}{a_{n}}}\right)\left|y_{n}\right|_{F}^{\frac{b_{n}-a_{n}+1}{a_{n}}}.
\end{align*}
Since $\frac{b_{n}+1}{a_{n}}<\frac{b_{j}+1}{a_{j}}$ we get for every
$1\leq j\leq n-1$, 
\[
a_{n}(b_{j}-a_{j}+1)>(b_{n}-a_{n}+1)a_{j},
\]
and hence: 
\[
b_{j}-a_{j}-\frac{(b_{n}-a_{n}+1)a_{j}}{a_{n}}>-1.
\]
In particular, integrating over the first $n-1$ coordinates, we get
that $\frac{d(f_{*}\mu)}{d\mu_{F}}(s)\lesssim\left|s\right|_{F}^{\frac{b_{n}-a_{n}+1}{a_{n}}}$.
If $\frac{b_{n}+1}{a_{n}}\geq1$ then $\epsilon_{\star}(f_{*}\mu)=\infty$
as required. Thus we may assume that $\frac{b_{n}+1}{a_{n}}<1$. Then
$f_{*}\mu\in\mathcal{M}_{c,1+\epsilon}(F)$, whenever $\frac{b_{n}-a_{n}+1}{a_{n}}(1+\epsilon)>-1$,
i.e.\ whenever 
\[
\epsilon<\frac{b_{n}+1}{a_{n}-b_{n}-1}=\frac{\frac{b_{n}+1}{a_{n}}}{1-\frac{b_{n}+1}{a_{n}}}=\frac{\min_{i}\frac{b_{i}+1}{a_{i}}}{1-\min_{i}\frac{b_{i}+1}{a_{i}}}.
\]
If $g(0)\neq0,$ we also have $\frac{d(f_{*}\mu)}{d\mu_{F}}(s)\gtrsim\left|s\right|_{F}^{\frac{b_{n}-a_{n}+1}{a_{n}}}$,
whence the inequality in the statement of the lemma is in fact an
equality.

It is left to deal with the case when $\min_{i}\frac{b_{i}+1}{a_{i}}$
is not uniquely achieved. For the lower bound, take 
\[
\widetilde{\mu}=\left|x_{1}\right|_{F}^{b_{1}}...\left|x_{n-1}\right|_{F}^{b_{n-1}}\left|x_{n}\right|_{F}^{b_{n}-\delta}g(x)\mu_{F}^{n},
\]
for an arbitrarily small $\delta>0$. Since $\widetilde{\mu}\geq\mu$
inside a small neighborhood of $0$, we deduce that $f_{*}\mu\in\mathcal{M}_{c,1+\epsilon}(F)$
for $\epsilon<\frac{\frac{b_{n}-\delta+1}{a_{n}}}{1-\frac{b_{n}-\delta+1}{a_{n}}}$.
Since $\delta$ can be taken arbitrarily small we are done. Similarly,
for the upper bound we take 
\[
\widetilde{\mu}=\left|x_{1}\right|_{F}^{b_{1}+\delta}\cdot...\cdot\left|x_{n-1}\right|_{F}^{b_{n-1}+\delta}\left|x_{n}\right|_{F}^{b_{n}}g(x)\mu_{F}^{n}\leq\mu,
\]
with $g(0)\neq0$ and deduce that $\epsilon_{\star}(f_{*}\mu)\leq\frac{\frac{b_{n}+1}{a_{n}}}{1-\frac{b_{n}+1}{a_{n}}}$. 
\end{proof}
\begin{proof}[Proof of Theorem \ref{thm:one dimensional}]
Let $\phi:X\to F$ be a locally dominant analytic map, and let $x_{0}\in X$.
Replacing $\phi$ with $\phi_{x_{0}}=\phi-\phi(x_{0})$, we may assume
that $\phi(x_{0})=0$. We may further assume that $X\subseteq F^{n}$
is open, and apply Theorem \ref{thm:analytic resolution of sing},
to get a log-principalization $\pi:\widetilde{X}\rightarrow X$, such
that, locally on a chart around a point in $\pi^{-1}(x_{0})$, $\phi_{x_{0}}\circ\pi(\widetilde{x}_{1},...,\widetilde{x}_{n})=C\widetilde{x}_{1}^{a_{1}}\cdot...\cdot\widetilde{x}_{n}^{a_{n}}$,
and 
\[
\operatorname{Jac}_{\widetilde{x}}(\pi)=v(\widetilde{x})\cdot\widetilde{x}_{1}^{b_{n}}\cdot...\cdot\widetilde{x}_{n}^{b_{n}}.
\]
Let $\sigma\in\mathcal{M}_{c,\infty}(X)$, with $\sigma=g(x)\mu_{F}^{n}$,
and $g>0$ in a neighborhood of $x_{0}$. Then $\sigma=\pi_{*}\mu$,
where 
\[
\mu=(g\circ\pi)(\widetilde{x}_{1},...,\widetilde{x}_{n})\left|v(\widetilde{x})\right|_{F}\left|\widetilde{x}_{1}\right|_{F}^{b_{1}}...\left|\widetilde{x}_{n}\right|_{F}^{b_{n}}\mu_{F}^{n}.
\]
Since $\phi_{*}\sigma=(\phi\circ\pi)_{*}\mu$, the theorem now follows
from Lemmas \ref{lem:formula for monomial maps} and \ref{Lemma:F-lct}. 
\end{proof}

\subsection{\label{sub:Further discussion}Relation to Fourier decay and other
invariants}

Let $\phi:X\rightarrow Y$ be an $F$-analytic map between $F$-analytic
manifolds $X$ and $Y$. We have seen that each of the exponents $\epsilon_{\star}(\phi;x)$
and $\operatorname{lct}_{F}(\phi_{x};x)$ provides a different quantification
for the singularities of $\phi$ near $x\in X$. When $Y=F^{m}$,
one can further consider other invariants involving the Fourier transform
of pushforward measures.

In $\S$\ref{subsec:Application:-regularization-by}, we have defined
$k_{\star}(\phi;x)$ as the minimal number of self-convolutions after
which the pushforward densities of smooth measures supported near
$x$ become bounded. Note that by the Plancherel theorem, for each
$\nu\in\mathcal{M}_{c,1}(F^{m})$ we have $\nu^{*k}\in\mathcal{M}_{c,2}(F^{m})$
if and only if $\mathcal{F}(\nu)\in L^{2k}(F^{m})$, whence 
\[
\nu^{*k}\in\mathcal{M}_{c,\infty}(F^{m})\Longrightarrow\mathcal{F}(\nu)\in L^{2k}(F^{m})\Longrightarrow\nu^{*2k}\in\mathcal{M}_{c,\infty}(F^{m}).
\]
Thus, the exponent $k_{\star}(\phi;x)$ is, in general, roughly comparable
to the $L^{p}$-class of $\mathcal{F}(\phi_{*}\mu)$ rather than to
the $L^{p}$-class of the pushforward measure $\phi_{*}\mu$ itself.
Instead of the $L^{p}$-class, we now focus on an invariant quantifying
the Fourier decay of $\phi_{*}\mu$ on the power-law scale: 
\[
\delta_{\star}(\phi;x):=\sup_{U\ni x}\inf_{\mu\in\mathcal{M}_{c,\infty}(U)}\delta_{\star}(\phi_{*}\mu),
\]
where 
\begin{equation}
\delta_{\star}(\nu):=\sup\{\delta\geq0\,:\,\left|\mathcal{F}(\nu)(y)\right|\lesssim\left\Vert y\right\Vert _{F}^{-\delta}\}.\label{eq:Fourier invariant}
\end{equation}
The study of the invariant $\delta_{\star}(\mu)$, and variations
of it, goes back at least to the 1920's, when the classical van der
Corput lemma was introduced, relating lower bounds on the derivative
of a smooth function $f:\R\rightarrow\R$, to bounds as in (\ref{eq:Fourier invariant}),
see \cite[Proposition 2]{ste93}, and \cite{CCW99}. This invariant
was also studied extensively in Igusa's work \cite{Igu78} in the
case $\dim Y=1$; it is much less understood in high dimensions. 
\begin{rem}
\label{rem:change of coordinates}Note that $\epsilon_{\star}(\phi;x)$
and $\operatorname{lct}_{F}(\phi_{x};x)$ are preserved under analytic
changes of coordinates around $x$. On the other hand, $\delta_{\star}(\phi;x)$
and $k_{\star}(\phi;x)$ might depend on the choice of coordinate
system. For example, the map $\phi(x,y)=(x,x^{2}(1+y^{1000}))$ satisfies
$\epsilon_{\star}(\phi)=\frac{1}{999}$ (by Theorem \ref{thm:lowerbd}),
while $\delta_{\star}(\phi)=\frac{1}{2}$ and $k_{\star}(\phi)\leq4$.
By applying the change of coordinates $\psi(x,y)=(x,y-x^{2})$, we
get $\psi\circ\phi(x,y)=(x,x^{2}y^{1000})$, and still have $\epsilon_{\star}(\psi\circ\phi)=\frac{1}{999}$,
while $\delta_{\star}(\psi\circ\phi)=\frac{1}{1000}$ and $k_{\star}(\psi\circ\phi)\geq1000$. 
\end{rem}

We next discuss the relations between the different exponents. In
the one-dimensional case, it turns out that all of the exponents above
are essentially equivalent (whenever $\operatorname{lct}_{F}(\phi;x)\leq1$).
In order to explain this, we need to discuss the structure of pushforward
measures by analytic maps.

\subsubsection{\label{subsec:Asymptotic-expansions-of}Asymptotic expansions of
pushforward measures and their Fourier transform}

Let $F$ be a local field, $f:U\rightarrow F$ be a locally dominant
analytic map with $U\subseteq F^{n}$ an open set. Let $\mu\in\mathcal{M}_{c}^{\infty}(U)$,
and consider the pushforward measure $f_{*}\mu$. Fix a non-trivial
additive character $\Psi$ of $F$. We may identify between $F$ and
$F^{\vee}$ by $t\longleftrightarrow\Psi_{t}$ where $\Psi_{t}(y)=\Psi(ty)$.
The Fourier transform $\mathcal{F}(f_{*}\mu)$ can then be written
as 
\[
\mathcal{F}(f_{*}\mu)(t)=\int_{F}\Psi(ty)df_{*}\mu(y)=\int_{U}\Psi(t\cdot f(x))d\mu(x).
\]
To $\mu$ and $f$, one can further associates \emph{Igusa's local
zeta function} 
\[
Z_{\mu,f}(s):=\int_{U}\left|f(x)\right|_{F}^{s}d\mu(x),\,\,\,\,\,\,s\in\C,\,\,\,\,\operatorname{Re}(s)>0.
\]
Igusa's local zeta function admits a meromorphic continuation to the
complex plane (see \cite{BG69,Ati70,Ber72} for the Archimedean case,
and \cite{Igu74,Igu78} for the non-Archimedean case). The poles of
$Z_{\mu,f}(s)$ (as well as certain twisted versions of it), and the
Laurent expansions around them, controls the asymptotic expansions
for $f_{*}\mu$ as $\left|y\right|_{F}\rightarrow0$ and for $\mathcal{F}(f_{*}\mu)$
when $\left|t\right|_{F}\rightarrow\infty$, via the theory of Mellin
transform (see \cite[Theorems 4.2, 4.3 and 5.3]{Igu78}). We next
describe the asymptotic expansions of $f_{*}\mu$ and $\mathcal{F}(f_{*}\mu)$.
\begin{defn}[{{Asymptotic expansion, see \cite[Section I.2]{Igu78}}}]
Let $x_{\infty}$ be $0$ or $\infty$. A sequence $\{\varphi_{k}\}_{k\in\N}$
of complex-valued functions on an open subset $U$ of $F\in\{\R,\C\}$,
with $\varphi_{k}(x)\neq0$ in a punctured neighborhood of $x_{\infty}$,
is said to constitute an \emph{asymptotic scale}, if for every $k$,
\[
{\displaystyle \lim\limits _{x\to x_{\infty}}\frac{\varphi_{k+1}(x)}{\varphi_{k}(x)}=0}.
\]
A function $f:U\rightarrow\C$ is said to have an \emph{asymptotic
expansion} at $x_{\infty}$, if there exists a sequence $\{a_{n}\}_{n\in\N}$
of complex numbers such that for every $k\geq0$, there exists $C>0$
such that for all $x$ close enough to $x_{\infty}$: 
\[
\left|f(x)-\sum\limits _{i=0}^{k}a_{i}\varphi_{i}(x)\right|\leq C\cdot\varphi_{k+1}(x).
\]
In this case, we write 
\begin{equation}
f(x)\approx\sum\limits _{k=0}^{\infty}a_{k}\varphi_{k}(x)\text{ as }x\rightarrow x_{\infty}.\label{eq:asymptotic}
\end{equation}
\end{defn}

\begin{example}
Given a monotone increasing sequence $\Lambda=(-1<\lambda_{0}<\lambda_{1}<...<\lambda_{n}<...)$
of real numbers, with no finite accumulation points, and given a sequence
$\{m_{n}\}_{n\in\N}$ of positive integers, set $\varphi_{0},\varphi_{1},\varphi_{2},...$
to be the sequence: 
\[
x^{\lambda_{0}}\log(x)^{m_{0}-1},x^{\lambda_{0}}\log(x)^{m_{0}-2},...,x^{\lambda_{0}},x^{\lambda_{1}}\log(x)^{m_{1}-1},...,x^{\lambda_{1}},...
\]
for $x>0$. Then $\{\varphi_{k}\}_{k\in\N}$ is an asymptotic scale
at $x_{\infty}=0$. 
\end{example}

We now describe the asymptotic expansions of pushforward measures
and their Fourier transforms. We fix a local field $F$ of characteristic
$0$ and an analytic $F$-manifold $X$. If $F$ is non-Archimedean
we further fix a uniformizer $\varpi_{F}\in\mathcal{O}_{F}$. We set
$F_{1}^{\times}:=\{x\in F:\left|x\right|_{F}=1\}$ and denote by $\ac:F^{\times}\rightarrow F_{1}^{\times}$
the angular component map
\[
\ac(t)=\begin{cases}
t\varpi_{F}^{-\val(t)} & \text{if }F\text{ is non-Archimedean}\\
\frac{t}{\left|t\right|} & \text{if }F\text{ is Archimedean}.
\end{cases}
\]

\begin{thm}[{\cite{Jea70,Mal74,Igu78}, see also \cite[Section 4]{VZG17} and \cite[Theorem 1.3.2 and Corollary 1.4.5]{Den91a}}]
\label{thm:asymptotic expansion}Let $f:X\rightarrow F$ be a locally
dominant analytic map, let $\mu\in\mathcal{M}_{c}^{\infty}(X)$, and
write $f_{*}\mu=g(y)d\mu_{F}$. Suppose that $0$ is the only critical
value of $f$. Then there exist: 
\begin{itemize}
\item a sequence $\Lambda=\{\lambda_{k}\}_{k\geq0}$, consisting of strictly
increasing positive real numbers with $\lim\limits _{k\to\infty}\lambda_{k}=\infty$,
if $F\in\{\R,\C\}$, or a finite set of complex numbers $\lambda_{k}$,
with $\operatorname{Re}\lambda_{k}>0$, and $\operatorname{Im}\lambda_{k}\in\frac{2\pi}{N\ln(q_{F})}\N$
for some $N\in\N$, if $F$ is non-Archimedean; 
\item a sequence $\{m_{k}\}_{k\geq0}$ of positive integers; 
\item smooth functions $a_{k,m,\mu},b_{k,m,\mu}$ on $F_{1}^{\times}$, 
\end{itemize}
such that: 
\begin{enumerate}
\item For $F\in\{\R,\C\}$, $g(y)$ admits an asymptotic expansion\footnote{The asymptotic expansion is also \emph{termwise differentiable,} and
\emph{uniform} in the angular component; we refer to \cite[p. 19-24]{Igu78}
for the precise meaning of those notions.} of the form 
\begin{equation}
g(y)\approx\sum_{k\geq0}\sum\limits _{m=1}^{m_{k}}a_{k,m,\mu}(\ac(y))\cdot\left|y\right|_{F}^{\lambda_{k}-1}(\log\left|y\right|_{F})^{m-1},\left|y\right|_{F}\rightarrow0,\label{eq:asympg}
\end{equation}
and $\mathcal{F}(f_{*}\mu)$ admits an asymptotic expansion of the
form, 
\begin{equation}
\mathcal{F}(f_{*}\mu)(t)\approx\sum_{k\geq0}\sum\limits _{m=1}^{m_{k}}b_{k,m,\mu}(\ac(t))\cdot\left|t\right|_{F}^{-\lambda_{k}}(\log\left|t\right|_{F})^{m-1},\left|t\right|_{F}\rightarrow\infty.\label{eq:Fourierasymp}
\end{equation}
\item For $F$ non-Archimedean, $g(y)$ and $\mathcal{F}(f_{*}\mu)$ admits
an expansion as in (\ref{eq:asympg}) and (\ref{eq:Fourierasymp}),
where \guillemotleft$\approx$\guillemotright{} is replaced with \guillemotleft$=$\guillemotright ,
and both sums are finite. 
\item For each $\lambda_{k}$, the functions $(b_{k,m,\mu})_{m=1}^{m_{k}}$
are determined by the functions $(a_{k,m,\mu})_{m=1}^{m_{k}}$. If
$\operatorname{Re}\lambda_{k}<1$, the map taking the latter to the
former is one-to-one, and, moreover, the leading function $b_{k,m_{k},\mu}$
is not identically zero (provided that $m_{k}$ is defined so that
$a_{k,m_{k},\mu}$ is not identically $0$). 
\end{enumerate}
\end{thm}

\begin{rem}
\label{rem:asymptotic expansion}~ 
\begin{enumerate}
\item The maps $(a_{k,m,\mu})_{m=1}^{m_{k}}\mapsto(b_{k,m,\mu})_{m=1}^{m_{k}}$
of item (3) can be explicitly described; see \cite[Proposition 4.6]{VZG17}
and \cite[Section 2.2]{Igu78}. 
\item Note that Igusa's theory (and in particular Theorem \ref{thm:asymptotic expansion})
was originally developed for polynomial maps but works for analytic
maps as well. Indeed, the proof uses resolution of singularities to
reduce to the case of pushforward of measures with monomial density
by monomial maps. The same reduction can be made for analytic maps
via an analytic version of resolution of singularities (as stated
in Theorem \ref{thm:analytic resolution of sing}). For a generalization
of Theorem \ref{thm:asymptotic expansion} to the case of meromorphic
maps, see \cite[Section 5]{VZG17}. 
\end{enumerate}
\end{rem}

We record the following immediate corollary of Equation (\ref{eq:asympg})
in Theorem \ref{thm:asymptotic expansion}: 
\begin{cor}
\label{cor:notach}Let $f:X\to F$ be a locally dominant analytic
map, let $\mu\in\mathcal{M}_{c}^{\infty}(X)$, and suppose that $\epsilon_{\star}(f_{*}\mu)<\infty$.
Then $f_{*}\mu\notin\mathcal{M}_{c,1+\epsilon_{\star}(f_{*}\mu)}(F)$,
i.e.\ the supremum in the definition of $\epsilon_{\star}$ is not
achieved. 
\end{cor}

Theorem \ref{thm:asymptotic expansion} also implies the following
corollary, relating $\epsilon_{\star}$ to $\delta_{\star}$: 
\begin{cor}
\label{cor:connection between epsilon and delta}Let $\phi:X\rightarrow F$
be a locally dominant analytic map such that $0$ is the only critical
value. Then for each $\mu\in\mathcal{M}_{c}^{\infty}(X)$ 
\begin{equation}
\epsilon_{\star}(\phi_{*}\mu)=\begin{cases}
\infty & \text{if }\delta_{\star}(\phi_{*}\mu)\geq1,\\
\frac{\delta_{\star}(\phi_{*}\mu)}{1-\delta_{\star}(\phi_{*}\mu)} & \text{if }\delta_{\star}(\phi_{*}\mu)<1.
\end{cases}\label{eq:epsilon and delta}
\end{equation}
\end{cor}

\begin{proof}
Write $\phi_{*}\mu=g(y)d\mu_{F}$. If $F=\R$, then by Theorem \ref{thm:asymptotic expansion},
$g(y)$ can be expanded near $y=0$, so that the leading term is $\left|y\right|_{F}^{\lambda_{k_{0}}-1}(\log\left|y\right|_{F})^{m_{k_{0}}-1}$
for some $\lambda_{k_{0}}\in\R_{>0}$ and $m_{k_{0}}\in\N$. Similarly,
$\mathcal{F}(\phi_{*}\mu)$ can be expanded near $t=\infty$. If $\lambda_{k_{0}}<1$,
then by Item (3) of Theorem \ref{thm:asymptotic expansion}, $\left|y\right|_{F}^{-\lambda_{k_{0}}}(\log\left|y\right|_{F})^{m_{k_{0}}-1}$
is the leading term of $\mathcal{F}(\phi_{*}\mu)$, and thus $\delta_{\star}(\phi_{*}\mu)=\lambda_{k_{0}}$.
If $\lambda_{k_{0}}=1$, then Item (3) implies that $\delta_{\star}(\phi_{*}\mu)\geq1$. 

Since $0$ is the only critical value of $\phi$, $g(y)$ is bounded
outside any neighborhood of $0$, so $g(y)^{1+\epsilon}$ is integrable
if and only if it is integrable in a small ball $B$ around $0$,
and this holds if and only if $(\lambda_{k_{0}}-1)(1+\epsilon)>-1$,
i.e.\ either $\lambda_{k_{0}}\geq1$ and then $\epsilon_{\star}(\phi_{*}\mu)=\infty$,
or $\epsilon<\frac{\lambda_{k_{0}}}{1-\lambda_{k_{0}}}=\frac{\delta_{\star}(\phi_{*}\mu)}{1-\delta_{\star}(\phi_{*}\mu)}$. 

The case when $F$ is non-Archimedean should be done with care, since
there might be multiple terms in (\ref{eq:asympg}) with the same
real part, and some cancellations may occur (see Example \ref{exa:cencalations}
below). Let $\lambda:=\underset{k}{\min}\operatorname{Re}\lambda_{k}$
and suppose $\lambda<1$. Then by Theorem \ref{thm:asymptotic expansion},
$g(y)$ can be written as $g(y)=g_{\mathrm{lead}}(y)+g_{\mathrm{error}}(y)$,
where
\[
g_{\mathrm{lead}}(y):=\sum_{m}\sum_{j=0}^{N-1}\widetilde{a}_{j,m,\mu}(\ac(y))\cdot\left|y\right|_{F}^{\lambda+\frac{2\pi ij}{N\ln(q_{F})}-1}(\log\left|y\right|_{F})^{m-1},
\]
$\widetilde{a}_{j,m,\mu}(\ac(y))$ is equal to $a_{k_{j},m,\mu}(\ac(y))$
if $\lambda_{k_{j}}=\lambda+\frac{2\pi ij}{N\ln(q_{F})}\in\Lambda$
and $0$ otherwise, and furthermore,
\[
g_{\mathrm{error}}(y):=\sum_{k:\operatorname{Re}\lambda_{k}>\lambda}\sum\limits _{m=1}^{m_{k}}a_{k,m,\mu}(\ac(y))\cdot\left|y\right|_{F}^{\lambda_{k}-1}(\log\left|y\right|_{F})^{m-1}.
\]
We first show there is an arithmetic progression $I_{a,N,l}:=\{a+bN\}_{l<b\in\N}$
for some $a,l\in\N$, and $y_{0}\in F_{1}^{\times}$, such that 
\begin{equation}
\left|g_{\mathrm{lead}}(y)\right|>C_{F}\left|y\right|_{F}^{\lambda-1},\label{eq:g is large on a large set}
\end{equation}
for all $y\in S_{a,N,l,d,y_{0}}:=\left\{ z\in\mathcal{O}_{F}:\val(\ac(z)-y_{0})\geq d,\,\val(z)\in I_{a,N,l}\right\} $,
for some constant $C_{F}$ independent of $y$. It is enough to show
(\ref{eq:g is large on a large set}) for the terms in $g_{\mathrm{lead}}(y)$
where $m=m_{0}$ is maximal such that for some $0\leq j\leq N-1$,
$\widetilde{a}_{j,m,\mu}(\ac(y))$ is not identically zero. Note that
\[
\left|y\right|_{F}^{\frac{2\pi ij}{N\ln(q_{F})}}=q_{F}^{-\val(y)\frac{2\pi ij}{N\ln(q_{F})}}=e^{-\ln(q_{F})\val(y)\frac{2\pi ij}{N\ln(q_{F})}}=e^{-\val(y)\frac{2\pi ij}{N}}.
\]
Let $y_{0}\in F_{1}^{\times}$ be such that $\widetilde{a}_{j,m_{0},\mu}(y_{0})\neq0$
for some $0\leq j\leq N-1$. Choose $d\in\N$ such that each of $\widetilde{a}_{0,m_{0},\mu}(z),\dots,\widetilde{a}_{0,m_{0},\mu}(z)$
is constant on the ball $\val(z-y_{0})\geq d$. Note that the functions
$f_{0},...,f_{N-1}:\Z/N\Z\rightarrow\C$, $f_{j}(t)=e^{-t\frac{2\pi ij}{N}}$
are the irreducible characters of $\Z/N\Z$ and hence they are linearly
independent. In particular, there exists $a\in\N$, such that
\[
\left|\sum_{j=0}^{N-1}\widetilde{a}_{j,m_{0},\mu}(y_{0})e^{-t\frac{2\pi ij}{N}}\right|=C_{F}>0,
\]
for all $t\in I_{a,N,0}$ . Taking $y\in S_{a,N,l,d,y_{0}}$ for $l\gg1$,
we deduce (\ref{eq:g is large on a large set}) as required. 

Finally, by Item (3) of Theorem \ref{thm:asymptotic expansion} and
by an argument similar to the one above, it follows that $\delta_{\star}(\phi_{*}\mu)=\lambda$.
Hence, (\ref{eq:g is large on a large set}) implies that $\int_{S_{a,N,l,d,y_{0}}}g(y)^{1+\epsilon}d\mu_{F}$
diverges for every $\epsilon\geq\frac{\lambda}{1-\lambda}=\frac{\delta_{\star}(\phi_{*}\mu)}{1-\delta_{\star}(\phi_{*}\mu)}$.
On the other hand, a similar argument as in the case $F=\R$ shows
that $\int_{F}g(y)^{1+\epsilon}d\mu_{F}$ converges for $\epsilon<\frac{\lambda}{1-\lambda}$,
so the corollary follows. 
\end{proof}
In the proof of Corollary \ref{cor:connection between epsilon and delta}
we have seen there might be some cancellations between the terms in
(\ref{eq:asympg}) with the same real part, and that these cancellations
are insignificant for infinitely many values of $\val(y)$. Here is
a simple example which illustrates this phenomenon. 
\begin{example}
\label{exa:cencalations}Let $d\in\N$ and let $p>d$ be a prime.
Let $\phi:\Qp\rightarrow\Qp$ be the map $\phi(x)=x^{d}$. Write $\phi_{*}\mu_{\Zp}=g(y)\mu_{\Zp}$.
Then for almost all $y\in\Zp$ we have
\[
g(y)=\#\phi^{-1}(y)\cdot\left|y\right|_{p}^{-1+\frac{1}{d}}=a(\ac(y))1_{\{z\in\Zp:d|\val(z)\}}(y)\cdot\left|y\right|_{p}^{-1+\frac{1}{d}}.
\]
where $a(z):=\#\{x\in\Zp:x^{d}=z\}$. Note that by Schur orthogonality,
$\sum_{j=0}^{d-1}e^{-t\frac{2\pi ij}{d}}=d$ if $d|t$ and is $0$
if $d\nmid t\in\N$. In particular
\[
1_{\{z\in\Zp:d|\val(z)\}}(y)=\frac{1}{d}\sum_{j=0}^{d-1}e^{-\val(y)\frac{2\pi ij}{d}}=\frac{1}{d}\sum_{j=0}^{d-1}\left|y\right|_{p}^{\frac{2\pi ij}{d\ln(q_{F})}},
\]
and therefore, the expansion of $g(y)$ as in (\ref{eq:asympg}) is
\[
g(y)=\sum_{j=0}^{d-1}\frac{a(\ac(y))}{d}\left|y\right|_{p}^{-1+\frac{1}{d}+\frac{2\pi ij}{d\ln(q_{F})}}.
\]
\end{example}

\subsubsection{\label{subsec:Relations-between-the}Relations between the invariants}

We now show that $\epsilon_{\star}(\phi;x),\delta_{\star}(\phi;x)$
and $k_{\star}(\phi;x)$ are essentially determined by the log-canonical
threshold $\operatorname{lct}_{F}(\phi_{x};x)$ whenever $\operatorname{lct}_{F}(\phi;x)\leq1$.
In particular, we show that 
\[
\epsilon_{\star}(\phi;x),\delta_{\star}(\phi;x),\frac{1}{k_{\star}(\phi;x)}\text{ and }\operatorname{lct}_{F}(\phi_{x};x)
\]
are asymptotically equivalent as $\operatorname{lct}_{F}(\phi_{x};x)\rightarrow0$.

Theorem \ref{thm:one dimensional} already shows that $\epsilon_{\star}(\phi;x)=\frac{\operatorname{lct}_{F}(\phi_{x};x)}{1-\operatorname{lct}_{F}(\phi_{x};x)}$
if $\operatorname{lct}_{F}(\phi_{x};x)<1$. We further have the following: 
\begin{prop}
\label{prop:Thom--Sebastiani}Let $\phi:X\rightarrow F$ be a dominant
$F$-analytic map. Then: 
\begin{enumerate}
\item If $\operatorname{lct}_{F}(\phi_{x};x)<1$, then $\delta_{\star}(\phi;x)=\operatorname{lct}_{F}(\phi_{x};x)$.
In particular, 
\[
\operatorname{lct}_{F}(\phi_{x}*\phi_{y};(x,y))\,\begin{cases}
=\operatorname{lct}_{F}(\phi_{x};x)+\operatorname{lct}_{F}(\phi_{y};y)~, & \operatorname{lct}_{F}(\phi_{x};x)+\operatorname{lct}_{F}(\phi_{y};y)<1,\\
\geq1~, & \text{otherwise}.
\end{cases}
\]
\item We have: 
\[
\left\lceil \frac{1}{\operatorname{lct}_{F}(\phi_{x};x)}\right\rceil \leq k_{\star}(\phi;x)\leq\left\lfloor \frac{1}{\operatorname{lct}_{F}(\phi_{x};x)}\right\rfloor +1.
\]
\end{enumerate}
\end{prop}

\begin{proof}
Let us first prove Item (1). Let $x\in X$. Replacing $\phi$ with
$\phi_{x}$, we may assume that $\phi(x)=0$. We may choose an open
neighborhood $U$ such that $0$ is the only critical value of $\phi|_{U}$,
$\overline{U}$ is compact and such that for each $\mu\in\mathcal{M}_{c}^{\infty}(U)$,
and each $s<\operatorname{lct}_{F}(\phi;x)$, one has $\int\left|\phi(x)\right|_{F}^{-s}d\mu(s)<\infty$.
Taking any $\mu$ which does not vanish at $x$, we get by Theorem
\ref{thm:one dimensional} and Lemma \ref{lem:formula for monomial maps}
that 
\[
\epsilon_{\star}(\phi_{*}\mu)=\frac{\operatorname{lct}_{F}(\phi_{x};x)}{1-\operatorname{lct}_{F}(\phi_{x};x)}=\epsilon_{\star}(\phi;x)<\infty.
\]
Corollary \ref{cor:connection between epsilon and delta} implies
that 
\[
\delta_{\star}(\phi_{*}\mu)=\frac{\epsilon_{\star}(\phi_{*}\mu)}{1+\epsilon_{\star}(\phi_{*}\mu)}=\frac{\epsilon_{\star}(\phi;x)}{1+\epsilon_{\star}(\phi;x)}=\operatorname{lct}_{F}(\phi_{x};x).
\]
Since the equalities above hold for $\mu$ of arbitrarily small support
around $x$, we get $\delta_{\star}(\phi;x)=\delta_{\star}(\phi_{*}\mu)=\operatorname{lct}_{F}(\phi_{x};x)$
as required. Since Fourier transform translates convolution into product,
we have 
\[
\delta_{\star}(\phi*\phi;(x,y))=\delta_{\star}(\phi;x)+\delta_{\star}(\phi;y),
\]
which implies the second part of Item (1) (see also \cite[Section 5.1]{Den91a}).

We now turn to Item (2). Set $k_{0}:=\left\lceil \frac{1}{\operatorname{lct}_{F}(\phi_{x};x)}\right\rceil $.
For a positive integer $k<k_{0}$ we get by Item (1), that $\operatorname{lct}_{F}(\phi_{x}^{*k};(x,...,x))<1$,
so that $\epsilon_{\star}(\phi_{x}^{*k};(x,...,x))<\infty$. This
implies the lower bound $k_{\star}(\phi;x)\geq k_{0}$. The upper
bound follows from (\ref{eq:via-young}) and Theorem \ref{thm:one dimensional}. 
\end{proof}
\begin{rem}
\label{rem:Analytic Thom--Sabestiani}Item (1) of Proposition \ref{prop:Thom--Sebastiani}
can be seen as an $F$-analytic interpretation of a theorem by Thom\textendash Sebastiani
\cite{STh71} (see e.g. \cite{MSS18} for the case $F=\C$). 
\end{rem}

The combination of Corollaries \ref{cor:notach} and \ref{cor:connection between epsilon and delta}
implies Theorem \ref{thm:reverse Young-intro}, as follows. 
\begin{proof}[Proof of Theorem \ref{thm:reverse Young-intro}]
Let $\mu_{j}\in\mathcal{M}_{c}^{\infty}(F^{n_{j}})$, and suppose
that 
\begin{equation}
\nu_{1}*\nu_{2}\in\mathcal{M}_{1+\epsilon}(F),\quad\epsilon>0,\label{eq:assum-revyoung}
\end{equation}
where $\nu_{j}=(\phi_{j})_{*}\mu_{j}$. Assume by contradiction that
\[
\frac{\epsilon_{\star}(\nu_{1})}{1+\epsilon_{\star}(\nu_{1})}+\frac{\epsilon_{\star}(\nu_{2})}{1+\epsilon_{\star}(\nu_{2})}\leq\frac{\epsilon}{1+\epsilon},
\]
For each $j\in\{1,2\}$, choose a finite cover $\bigcup_{i}U_{i,j}$
of $\operatorname{Supp}(\mu_{j})$ by open balls in $F^{n_{j}}$,
such that for each $i$, $\phi_{j}|_{U_{i,j}}$ has at most one critical
value $z_{i,j}$ of $\phi_{j}$. We can write $\mu_{j}=\sum_{i}\mu_{i,j}$
with $\mu_{i,j}$ supported inside $U_{i,j}$. Taking $\phi_{i,j}=\phi_{j}|_{U_{i,j}}-z_{i,j}$,
we can find $i_{1}$ and $i_{2}$ such that 
\[
\epsilon_{\star}(\nu_{j})=\epsilon_{\star}\left((\phi_{i_{j},j})_{*}(\mu_{i_{j},j})\right),\quad j=1,2.
\]
Now, Corollary \ref{cor:connection between epsilon and delta} implies
that 
\[
\delta_{\star}\left((\phi_{i_{j},j})_{*}(\mu_{i_{j},j})\right)=\frac{\epsilon_{\star}\left((\phi_{i_{j},j})_{*}(\mu_{i_{j},j})\right)}{1+\epsilon_{\star}\left((\phi_{i_{j},j})_{*}(\mu_{i_{j},j})\right)},
\]
whence 
\[
\delta_{\star}\left((\phi_{i_{1},1})_{*}(\mu_{i_{1},1})*(\phi_{i_{2},2})_{*}(\mu_{i_{2},2})\right)\leq\frac{\epsilon}{1+\epsilon}.
\]
Using Corollary \ref{cor:connection between epsilon and delta} once
again, we obtain that 
\[
\epsilon_{\star}(\nu_{1}*\nu_{2})\leq\epsilon_{\star}\left((\phi_{i_{1},1})_{*}(\mu_{i_{1},1})*(\phi_{i_{2},2})_{*}(\mu_{i_{2},2})\right)\leq\epsilon.
\]
In view of Corollary \ref{cor:notach}, this contradicts (\ref{eq:assum-revyoung}). 
\end{proof}

\subsection{Consistency of the various bounds in the one-dimensional case}

The following proposition, which can be seen as a variant of the \L{}ojasiewicz
gradient inequality (see e.g. \cite[p.92]{Loj65}, \cite[Proposition 6.8]{BM88}),
shows that the formula for $\epsilon_{\star}(\phi;x)$ given in Theorem
\ref{thm:one dimensional} in the one-dimensional case is consistent
with the lower and upper bounds in Theorems \ref{thm:lowerbd} and
\ref{thm:upper-C}. The proof of Proposition \ref{prop:Lbounds agree with formula}
is similar to \cite[Theorem 1]{Fee19}, but applied to any local field.

Note that for $\phi:F^{n}\rightarrow F$, we have $\mathcal{J}_{\phi}=\langle\frac{\partial\phi}{\partial x_{1}},...,\frac{\partial\phi}{\partial x_{n}}\rangle$,
which we denote by $\langle\nabla\phi\rangle$. 
\begin{prop}
\label{prop:Lbounds agree with formula}Let $\phi:F^{n}\rightarrow F$
be an analytic map. Then for every $x\in F^{n}$, we have 
\[
\frac{\operatorname{lct}_{F}(\langle\nabla\phi\rangle;x)}{1-\operatorname{lct}_{F}(\langle\nabla\phi\rangle;x)}\geq\frac{\operatorname{lct}_{F}(\phi_{x};x)}{1-\operatorname{lct}_{F}(\phi_{x};x)}\geq\operatorname{lct}_{F}(\langle\nabla\phi\rangle;x)\geq\operatorname{lct}_{F}(\phi_{x};x),
\]
where the middle (resp. left) inequality holds whenever $\operatorname{lct}_{F}(\phi_{x};x)<1$
(resp. $\operatorname{lct}_{F}(\langle\nabla\phi\rangle;x)<1$). 
\end{prop}

\begin{proof}
The middle inequality follows from Theorems \ref{thm:lowerbd} and
\ref{thm:one dimensional}, and the left inequality follows from the
right inequality, so it is left to show that $\operatorname{lct}_{F}(\langle\nabla\phi\rangle;x_{0})\geq\operatorname{lct}_{F}(\phi_{x_{0}};x_{0})$
for a fixed $x_{0}\in F^{n}$.

Recall that $\left\Vert (a_{1},...,a_{n})\right\Vert _{F}:=\max_{i}\left|a_{i}\right|_{F}$
is the maximum norm on $F^{n}$. We would like to relate between $\int_{B}\left\Vert \nabla\phi\right\Vert _{F}^{-s}dx$
and $\int_{B}\left|\phi_{x_{0}}\right|_{F}^{-s}dx$, where $B$ is
a small ball around $x_{0}$, for $s>0$ small enough. Let $\pi:\widetilde{X}\rightarrow F^{n}$
be a resolution of singularities of $\phi_{x_{0}}$. Then we have
\begin{equation}
\int_{B}\left\Vert \nabla\phi\right\Vert _{F}^{-s}dx=\int_{\pi^{-1}(B)}\left\Vert \nabla\phi|_{\pi(\widetilde{x})}\right\Vert _{F}^{-s}\left|\operatorname{Jac}_{\widetilde{x}}(\pi)\right|_{F}d\widetilde{x}.\label{eq:star}
\end{equation}
Since $\pi^{-1}(B)$ is compact, by working locally over finitely
many pieces, and using Theorem \ref{thm:analytic resolution of sing},
we may replace $\pi^{-1}(B)$ by a compact neighborhood $L\subseteq F^{n}$
of $0$, $\pi(0)=x_{0}$, and further assume that for $\widetilde{x}=(\widetilde{x}_{1},...,\widetilde{x}_{n})$:
\begin{equation}
\widetilde{\phi}(\widetilde{x}):=\phi\circ\pi(\widetilde{x})=C\widetilde{x}_{1}^{a_{1}}\cdots\widetilde{x}_{n}^{a_{n}}\text{ and }\operatorname{Jac}_{\widetilde{x}}(\pi)=v(\widetilde{x})\widetilde{x}_{1}^{b_{1}}\cdots\widetilde{x}_{n}^{b_{n}},\label{eq:resolution data}
\end{equation}
for some analytic unit $v(\widetilde{x})$ and a constant $C\in F$.
Note that on $L$ all of the entries of $d_{\widetilde{x}}\pi$ are
smaller, in absolute value, than some constant $\widetilde{C}$, so
that the operator norm $\left\Vert d_{\widetilde{x}}\pi\right\Vert _{\mathrm{op}}:=\underset{v\in F^{n}:\left\Vert v\right\Vert _{F}=1}{\sup}\left\Vert d_{\widetilde{x}}\pi\cdot v\right\Vert _{F}$
of $d_{\widetilde{x}}\pi$ is bounded by a constant. Since $\nabla\widetilde{\phi}|_{\widetilde{x}}=\nabla\phi|_{\pi(\widetilde{x})}\cdot d_{\widetilde{x}}\pi$,
for $s>0$ we get: 
\begin{align*}
\text{\ensuremath{\int_{L}\left\Vert \nabla\phi|_{\pi(\widetilde{x})}\right\Vert _{F}^{-s}\left|\operatorname{Jac}_{\widetilde{x}}(\pi)\right|_{F}}}d\widetilde{x} & \leq\int_{L}\left\Vert d_{\widetilde{x}}\pi\right\Vert _{\mathrm{op}}^{s}\left\Vert \nabla\widetilde{\phi}|_{\widetilde{x}}\right\Vert _{F}^{-s}\left|\operatorname{Jac}_{\widetilde{x}}(\pi)\right|_{F}d\widetilde{x}\lesssim\int_{L}\left\Vert \nabla\widetilde{\phi}|_{\widetilde{x}}\right\Vert _{F}^{-s}\left|\operatorname{Jac}_{\widetilde{x}}(\pi)\right|_{F}d\widetilde{x}\\
 & \lesssim\int_{L}\left(\max_{j}\left|\widetilde{x}_{1}\right|_{F}^{a_{1}}\cdot...\cdot\left|\widetilde{x}_{j}\right|_{F}^{a_{j}-1}\cdot...\cdot\left|\widetilde{x}_{n}\right|_{F}^{a_{n}}\right)^{-s}\left|\operatorname{Jac}_{\widetilde{x}}(\pi)\right|_{F}d\widetilde{x}\\
 & \lesssim\int_{L}\left(\left|\widetilde{x}_{1}\right|_{F}^{a_{1}}...\left|\widetilde{x}_{n}\right|_{F}^{a_{n}}\right)^{-s}\left|\operatorname{Jac}_{\widetilde{x}}(\pi)\right|_{F}d\widetilde{x}\lesssim\int_{\pi(L)}\left|\phi_{x_{0}}\right|_{F}^{-s}dx,
\end{align*}
which concludes the proof. 
\end{proof}
\begin{rem}
Here is an alternative approach to Proposition \ref{prop:Lbounds agree with formula},
as suggested by the referee. For simplicity suppose $F=\C$. We may
assume that $x=0\in\C^{n}$. Recall from \cite[Definition 1.1.1]{HS06}
that the integral closure $\overline{I}$ of an ideal $I$ in a commutative
ring $R$, is the set of elements $r\in R$ for which there are $a_{j}\in I^{j}$
such that $\sum_{j=0}^{n}a_{j}r^{n-j}=0$ for some $n\in\N$. We now
take $R=\C\{x_{1},...,x_{n}\}$ to be the ring of convergent power
series, $I:=\langle\nabla\phi\rangle$ and $J:=\langle\phi_{0}\rangle$.
By \cite[Corollary 7.1.4]{HS06}, we have $J\subseteq\overline{I}$
and in particular $\operatorname{lct}_{\C}(J;0)\leq\operatorname{lct}_{\C}(\overline{I};0)$
(see \cite[Property 1.15]{Mus12}). However, by \cite[Property 1.12]{Mus12}
(see also \cite[Proposition 2.4]{dFM09}) we have $\operatorname{lct}_{\C}(I;0)=\operatorname{lct}_{\C}(\overline{I};0)$,
which implies that $\operatorname{lct}_{\C}(J;0)\leq\operatorname{lct}_{\C}(I;0)$
as required. This argument likely generalizes to any local field $F$
of characteristic $0$, i.e. the full generality of Proposition \ref{prop:Lbounds agree with formula}. 
\end{rem}

\subsection{\label{subsec:Higher-dimensions}Some examples in higher dimension}

We next discuss the higher dimensional case (i.e.~$\dim Y>1$). Here
we will see that the connection between the four invariants $\epsilon_{\star}(\phi;x)$,
$\delta_{\star}(\phi;x)$, $k_{\star}(\phi;x)$ and $\operatorname{lct}_{F}(\phi_{x};x)$
is not as tight as in the one-dimensional case. We first provide a
simpler description of $\delta_{\star}(\phi;x)$. 
\begin{lem}
\label{lem:reduction to one dimensional case}Let $F$ be a local
field of characteristic $0$ and $\phi:F^{n}\rightarrow F^{m}$ be
an analytic map. Then 
\[
\delta_{\star}(\phi;x)=\inf_{\ell}\{\delta_{\star}(\ell\circ\phi;x)\},
\]
where $\ell$ runs over all non-zero linear functionals $\ell:F^{m}\rightarrow F$. 
\end{lem}

\begin{proof}
For each $\mu\in\mathcal{M}_{c}^{\infty}(F^{n})$, we have 
\[
\left|\mathcal{F}(\phi_{*}\mu)(z)\right|\lesssim\left\Vert z\right\Vert _{F}^{-\delta}\Longleftrightarrow\left|\mathcal{F}(\phi_{*}\mu)(ta_{1},...,ta_{m})\right|\lesssim\left|t\right|_{F}^{-\delta},
\]
for each $a=(a_{1},...,a_{m})$ with $\left\Vert a\right\Vert _{F}=1$
and $t\in F$. Setting $\ell_{a}(y_{1},...,y_{m})=\sum_{i=1}^{m}a_{i}y_{i}$,
we have, 
\[
\mathcal{F}(\phi_{*}\mu)(ta_{1},...,ta_{m})=\varint_{F^{m}}\Psi(t\ell_{a}(y))\cdot d\phi_{*}\mu=\varint_{F}\Psi(tz)\cdot d(\ell_{a}\circ\phi)_{*}\mu=\mathcal{F}\left((\ell_{a}\circ\phi)_{*}\mu\right)(t),
\]
which concludes the proof of the lemma. 
\end{proof}
The following examples demonstrate that the connection between the
four invariants can get loose as the dimension of $Y$ grows. 
\begin{example}
\label{exa:no connection when m>1}Consider the map $\phi:\C^{m}\rightarrow\C^{m}$,
defined by $\phi(x_{1},...,x_{m})=(x_{1}^{d},x_{1}^{d}x_{2},...,x_{1}^{d}x_{m})$.
Then: 
\begin{enumerate}
\item $\epsilon_{\star}(\phi)=\frac{1}{dm-1}$. 
\item $\operatorname{lct}_{\C}(J)=\frac{1}{d}$, where $J=\langle x_{1}^{d},x_{1}^{d}x_{2},...,x_{1}^{d}x_{m}\rangle=\langle x_{1}^{d}\rangle$. 
\item $dm\leq k_{\star}(\phi)\leq dm+1$. 
\item $\delta_{\star}(\phi)=\frac{1}{d}$. 
\end{enumerate}
Item (1) follows from Theorem \ref{thm:lowerbd}. Young's inequality
(see \ref{eq:via-young}) shows that $k_{\star}(\phi)\leq dm+1$,
on the other hand we have $\operatorname{lct}_{\C}(J)=\frac{1}{d}$
which implies that at least $dm$ convolutions are needed to obtain
the (FRS) property (see \cite[Lemmas 3.23 and 3.26]{GHb}), hence
$k_{\star}(\phi)\geq dm$. Note that for any $a_{1},..,a_{m}\in\C$
with $\left\Vert a\right\Vert _{\C}=1$ we have $\operatorname{lct}_{\C}(a_{1}x_{1}^{d}+a_{2}x_{1}^{d}x_{2}+...+a_{m}x_{1}^{d}x_{m})=\frac{1}{d}$
so $\delta_{\star}(\phi)=\frac{1}{d}$, by Lemma \ref{lem:reduction to one dimensional case}. 
\end{example}

\begin{rem}
One can replace $\phi$ in Example \ref{exa:no connection when m>1},
with 
\[
\phi(x_{1},...,x_{m})=(x_{1},x_{1}^{d}x_{2},...,x_{1}^{d}x_{m}).
\]
Here, we get $\epsilon_{\star}(\phi)=\frac{1}{dm-d}$, which is similar
to Example \ref{exa:no connection when m>1} when $m$ is large, while
$\operatorname{lct}_{\C}(J)=1$, where $J=\langle x_{1},x_{1}^{d}x_{2},...,x_{1}^{d}x_{m}\rangle=\langle x_{1}\rangle$,
which is much larger than in Example \ref{exa:no connection when m>1}. 
\end{rem}

The following example shows that a reverse Young inequality (Theorem
\ref{thm:reverse Young-intro}) does not hold in dimension $m\geq2$. 
\begin{example}
\label{exa:no reverse Young}Consider $\phi(x,y)=(x,x^{2}(1+y^{1000}))$.
Then $\epsilon_{\star}(\phi)=\frac{1}{999}$, while $\delta_{\star}(\phi)=\frac{1}{2}$,
and consequently $k_{\star}(\phi)\leq4$. 
\end{example}

\section{\label{sec:lower bound}Lower bound: proof of Theorem \ref{thm:lowerbd}}

In this section we prove Theorem~\ref{thm:lowerbd}. We start with
the equidimensional case, which we restate as Proposition~\ref{prop:equal dimension}
below.

Recall that given a locally dominant analytic map $\phi:X\to Y$ between
two $F$-analytic manifolds, the Jacobian ideal sheaf $\mathcal{J}_{\phi}$
is defined locally as the ideal in the algebra of analytic functions
on $X$ generated by the $m\times m$-minors of the differential $d_{x}(\phi)$
of $\phi$. 
\begin{prop}
\label{prop:equal dimension}Let $X,Y$ be $F$-analytic manifolds,
$\dim X=\dim Y=n$, and let $\phi:X\to Y$ be a locally dominant analytic
map. Then for any $x_{0}\in X$, 
\begin{equation}
\epsilon_{\star}(\phi;x_{0})=\operatorname{lct}_{F}(\mathcal{J}_{\phi};x_{0}).\label{eq:formula equal dimension}
\end{equation}
\end{prop}

\begin{proof}
We may assume that $X,Y$ are compact balls in $F^{n}$, and $x_{0}=0$.
Since $\dim X=\dim Y=n$, we have $\mathcal{J}_{\phi}=\langle\operatorname{Jac}_{x}(\phi)\rangle$,
where $\operatorname{Jac}_{x}(\phi)=\det\,(d_{x}\phi)$. Since $\phi$
is analytic and $X$ is compact, there is an open dense set $U\subseteq X$
and $M\in\N$ such that $\#\left\{ \phi^{-1}\phi(x)\right\} \leq M$
and $\operatorname{Jac}_{x}(\phi)\neq0$, for every $x\in U$. We
choose a disjoint cover $\bigcup_{i\in\N}U_{i}$ of $U$ by locally
closed subsets\footnote{Recall that a set is locally closed if it is the difference of two
open sets.} $U_{i}$ such that $\phi|_{U_{i}}$ is a diffeomorphism. Write $\mu:=\mu_{F}^{n}|_{X}$,
$\mu_{i}:=\mu_{F}^{n}|_{U_{i}}$, and note that $\phi_{*}\mu=g(y)\cdot\mu_{F}^{n}$
and $\phi_{*}\mu_{i}=g_{i}(y)\cdot\mu_{F}^{n}$, where 
\[
g(y)=\sum_{x\in\phi^{-1}(y)}\left|\operatorname{Jac}_{x}(\phi)\right|_{F}^{-1}\text{ and }g_{i}(y)=\left|\operatorname{Jac}_{\phi^{-1}(y)}(\phi)\right|_{F}^{-1}.
\]
We have 
\begin{equation}
\int_{Y}g(y)^{1+s}dy=\int_{\phi(U)}g(y)^{s}d\phi_{*}\mu=\int_{U}g\circ\phi(x)^{s}d\mu\geq\int_{U}\frac{1}{\left|\operatorname{Jac}_{x}(\phi)\right|_{F}^{s}}d\mu.\label{eq:formula for g^1+s equal dim}
\end{equation}
On the other hand, since $\#\left\{ i\in I:y\in\phi(U_{i})\right\} \leq M$
for each $y\in Y$, using Jensen's inequality, we have: 
\begin{align}
\int_{Y}g(y)^{1+s}dy & \leq\int_{Y}\left(\sum_{i\in I:y\in\phi(U_{i})}g_{i}(y)\right)^{1+s}dy\leq M^{s}\int_{Y}\sum_{i\in I:y\in\phi(U_{i})}g_{i}(y)^{1+s}dy\nonumber \\
 & =M^{s}\sum_{i\in I}\int_{Y}g_{i}(y)^{1+s}dy=M^{s}\sum_{i\in I}\int_{U_{i}}\frac{1}{\left|\operatorname{Jac}_{x}(\phi)\right|_{F}^{s}}d\mu=M^{s}\int_{U}\frac{1}{\left|\operatorname{Jac}_{x}(\phi)\right|_{F}^{s}}d\mu,\label{eq:formula for epsilon m=00003Dn}
\end{align}
which implies the proposition. 
\end{proof}
\begin{proof}[Proof of Theorem \ref{thm:lowerbd}]
Let $\phi:X\to Y$ be a locally dominant analytic map, and let $x_{0}\in X$.
Since the claim is local, we may assume that $X\subseteq F^{n}$ is
an open subset, and $Y=F^{m}$, with $n\geq m$. Let $B$ be a small
ball around $x_{0}$. For each subset $I\subseteq\{1,...,n\}$ of
size $m$, let $M_{I}$ be the corresponding $m\times m$-minor of
$d_{x}\phi$, and set 
\[
V_{I}:=\left\{ x\in B:\max_{I'\subseteq\{1,...,n\}}\left|M_{I'}(x)\right|_{F}=\left|M_{I}(x)\right|_{F}\right\} .
\]
If $V_{I}\neq\varnothing$, the map $\phi_{I}(x):=(\phi(x),x_{j_{1}},...,x_{j_{n-m}})$
is locally dominant, where $\operatorname{Jac}_{x}(\phi_{I})=M_{I}(x)$,
and $\{j_{1},...,j_{n-m}\}=\{1,...,n\}\backslash I$. For each $V_{I}$
of positive measure, let $\mu_{B}:=\mu_{F}^{n}|_{B}$, $\mu_{I}:=\mu_{F}^{n}|_{B\cap V_{I}}$
and write 
\[
\phi_{*}\mu_{B}=g(y)\cdot\mu_{F}^{m},\,\,\,\phi_{*}\mu_{I}=g_{I}(y)\cdot\mu_{F}^{m}\text{ and }\left(\phi_{I}\right)_{*}\mu_{I}=\widetilde{g}_{I}(z)\cdot\mu_{F}^{n}.
\]
Let $\widetilde{B}$ be a large ball in $F^{n-m}$ which contains
the projection of $\phi_{I}(B\cap V_{I})$ from $F^{n}$ to the last
$n-m$ coordinates $F^{n-m}$. Since $\phi=\pi_{I}\circ\phi_{I}$
where $\pi_{I}:F^{n}\rightarrow F^{m}$ is a projection to the first
$m$ coordinates, we have 
\[
g_{I}(y)=\int_{F^{n-m}}\widetilde{g}_{I}(y,z_{m+1},...,z_{n})dz_{m+1}...dz_{n}=\int_{\widetilde{B}}\widetilde{g}_{I}(y,z_{m+1},...,z_{n})dz_{m+1}...dz_{n}.
\]
By Jensen's inequality, we have
\begin{align*}
\int_{F^{m}}g_{I}(y)^{1+s}dy= & \int_{F^{m}}\mu_{F}^{n-m}(\widetilde{B})^{1+s}\left(\frac{1}{\mu_{F}^{n-m}(\widetilde{B})}\int_{\widetilde{B}}\widetilde{g}_{I}(y,z_{m+1},...,z_{n})dz_{m+1}...dz_{n}\right)^{1+s}dy\\
\leq & \mu_{F}^{n-m}(\widetilde{B})^{s}\cdot\int_{F^{m}\times\widetilde{B}}\widetilde{g}_{I}(y,z_{m+1},...,z_{n})^{1+s}dydz_{m+1}...dz_{n}\lesssim\int_{F^{n}}\widetilde{g}_{I}(z)^{1+s}dz,
\end{align*}
for every $s>0$. Since $\bigcup_{I}V_{I}$ is of full measure in
$B$, and using Proposition \ref{prop:equal dimension} and (\ref{eq:lct-analytic definition}),
we have: 
\begin{align*}
\int_{F^{m}}g(y)^{1+s}dy & \lesssim\sum_{I}\int_{F^{m}}g_{I}(y)^{1+s}dy\lesssim\sum_{I}\int_{F^{n}}\widetilde{g}_{I}(z)^{1+s}dz\\
 & \lesssim\sum_{I}\int_{V_{I}}\left|M_{I}(x)\right|_{F}^{-s}dx\lesssim\int_{B}\min_{I\subseteq\{1,...,n\}}[\left|M_{I}(x)\right|_{F}^{-s}]dx<\infty,
\end{align*}
for every $s<\operatorname{lct}_{F}(\mathcal{J}_{\phi};x_{0})$, as
required. 
\end{proof}

\section{\label{sec:upper bound}Proof of Theorem \ref{thm:upper-C} \textendash{}
an upper bound over $\mathbb{C}$}

In this section we prove Theorem~\ref{thm:upper-C}. We use the following
easy consequence of the coarea formula (see \cite{Fed69}). 
\begin{lem}
\label{lem:coarea}Let $B\subseteq\C^{n}$ be a compact ball, and
let $\phi:B\to\C^{m}$ be a dominant analytic map. Let $\mu_{B}:=\mu_{\C}^{n}|_{B}$
and write $\phi_{*}\mu_{B}=g(y)\cdot\mu_{\C}^{m}$. Then: 
\[
g(y)=\int_{\phi^{-1}(y)}\frac{dH_{2(n-m)}(x)}{\det\left((d_{x}\phi)^{*}(d_{x}\phi)\right)},
\]
where the integral is taken with respect to the $2(n-m)$-dimensional
Hausdorff measure (recall $\mathsection$\ref{subsec:Conventions}(7)). 
\end{lem}

Let $\phi:X\rightarrow Y$ be a locally dominant analytic map between
complex analytic manifolds, and let $x_{0}\in X$. Since the claim
is local, we may assume that $Y=\C^{m}$ and $X=B$ is a ball in $\C^{n}$.
To show (\ref{eq:upper bound on epsilon}), it is enough to bound
$\epsilon_{\star}(\phi_{*}\mu_{B})$, where $B$ is of arbitrarily
small radius around $x_{0}$. Denote $G(x):=\det\left((d_{x}\phi)^{*}(d_{x}\phi)\right)$.
By Lemma \ref{lem:coarea}, we have: 
\[
g(y)=\int_{\phi^{-1}(y)}\frac{dH_{2(m-n)}(x)}{G(x)}
\]
whence 
\begin{equation}
\int_{Y}g(y)^{1+\epsilon}dy=\int_{X}\left[\int_{\phi^{-1}(\phi(x))}\frac{dH_{2(m-n)}(x')}{G(x')}\right]^{\epsilon}dx.\label{eq:overC}
\end{equation}

We apply Theorem \ref{thm:analytic resolution of sing} to the ideal
$\mathcal{J}_{\phi}$. Let $\pi:\widetilde{X}\to B$ be the corresponding
resolution. Without loss of generality, we can assume that $\widetilde{X}\subseteq\mathbb{C}^{m}$
is an open subset, so that 
\[
\det d_{\widetilde{x}}\pi=v(\widetilde{x})\widetilde{x}_{1}^{b_{1}}\cdots\widetilde{x}_{n}^{b_{n}}~,\quad v(0)\neq0,
\]
and each of the $m\times m$ minors $M_{I}$ of $d_{\widetilde{x}}\phi$
satisfies 
\[
M_{I}(\pi(\widetilde{x}))=u_{I}(\widetilde{x})\widetilde{x}_{1}^{a_{1}}\cdots\widetilde{x}_{n}^{a_{n}},
\]
where at least one of the functions $u_{I}$ does not vanish at $0$.

We first perform the change of variables $x=\pi(\widetilde{x})$ in
the external integral in (\ref{eq:overC}), yielding 
\[
\int_{Y}g(y)^{1+\epsilon}dy=\int_{\widetilde{X}}\left|\det d_{\widetilde{x}}\pi\right|_{\C}\left[\int_{\phi^{-1}(\phi(\pi(\widetilde{x})))}\frac{dH_{2(m-n)}(x')}{G(x')}\right]^{\epsilon}d\widetilde{x},
\]
and then perform the change of variables $x'=\pi(t)$ in the internal
integral. Since the map $\pi$ is continuously differentiable, and
$B$ is compact, the product of the singular values of the restriction
of $d_{t}\pi$ to any subspace is bounded from below in absolute value
by a number times the absolute value of the determinant of $d_{t}\pi$.
Thus, 
\begin{equation}
\int_{Y}g(y)^{1+\epsilon}dy\gtrsim\int_{\widetilde{X}}\left|\det d_{\widetilde{x}}\pi\right|_{\C}\left[\int_{\widetilde{\phi}^{-1}(\widetilde{\phi}(\widetilde{x}))}\frac{\left|\det d_{t}\pi\right|_{\C}\,dH_{2(m-n)}(t)}{G(\pi(t))}\right]^{\epsilon}d\widetilde{x},\label{eq:int-tilde}
\end{equation}
where $\widetilde{\phi}:=\phi\circ\pi$. By the Cauchy\textendash Binet
formula, the right-hand side is equal to 
\[
\int_{\widetilde{X}}\left|\det d_{\widetilde{x}}\pi\right|_{\C}\left[\int_{\widetilde{\phi}^{-1}(\widetilde{\phi}(\widetilde{x}))}\frac{\left|v(t)\right|_{\C}\left|t_{1}\right|_{\C}^{b_{1}}\cdot...\cdot\left|t_{n}\right|_{\C}^{b_{n}}\,dH_{2(m-n)}(t)}{\sum_{I}\left|u_{I}(t)\right|_{\C}\left|t_{1}\right|_{\C}^{a_{1}}\cdot...\cdot\left|t_{n}\right|_{\C}^{a_{n}}}\right]^{\epsilon}d\widetilde{x}.
\]
We bound the internal integral from below as follows. For each $\widetilde{x}\in\widetilde{X}$,
set 
\[
Q(\widetilde{x}):=\left\{ t\in\widetilde{X}\,\,:\,\,\sum_{j=1}^{n}\frac{\left|t_{j}-\widetilde{x}_{j}\right|_{\C}}{\left|\widetilde{x}_{j}\right|_{\C}}\leq\frac{1}{4}\right\} 
\]
Then 
\begin{align}
 & \int_{\widetilde{\phi}^{-1}(\widetilde{\phi}(\widetilde{x}))}\frac{\left|v(t)\right|_{\C}\left|t_{1}\right|_{\C}^{b_{1}}\cdot...\cdot\left|t_{n}\right|_{\C}^{b_{n}}\,dH_{2(m-n)}(t)}{\sum_{I}\left|u_{I}(t)\right|_{\C}\left|t_{1}\right|_{\C}^{a_{1}}\cdot...\cdot\left|t_{n}\right|_{\C}^{a_{n}}}\nonumber \\
\gtrsim & \int_{\widetilde{\phi}^{-1}(\widetilde{\phi}(\widetilde{x}))\cap Q(\widetilde{x})}\left|t_{1}\right|_{\C}^{b_{1}-a_{1}}\cdot...\cdot\left|t_{n}\right|_{\C}^{b_{n}-a_{n}}\,dH_{2(m-n)}(t)\nonumber \\
\gtrsim & \left|\widetilde{x}_{1}\right|_{\C}^{b_{1}-a_{1}}\cdot...\cdot\left|\widetilde{x}_{n}\right|_{\C}^{b_{n}-a_{n}}\,H_{2(n-m)}(\widetilde{\phi}^{-1}(\widetilde{\phi}(\widetilde{x}))\cap Q(\widetilde{x})).\label{eq:lower bound}
\end{align}

\begin{claim}
\label{cl:lelong}For any $\widetilde{x}\in\widetilde{X}$, 
\[
H_{2(n-m)}(\widetilde{\phi}^{-1}(\widetilde{\phi}(\widetilde{x}))\cap Q(\widetilde{x}))\gtrsim\min_{|I|=n-m}\prod_{i\in I}\left|\widetilde{x}_{i}\right|_{\C}.
\]
\end{claim}

\begin{proof}[Proof of Claim~\ref{cl:lelong}]
Let $T:\widetilde{X}\to\widetilde{X}$ be the affine map given by
$T(\widetilde{x}')_{j}=\widetilde{x}_{j}(1+\widetilde{x}'_{j})$.
Then $Q(\widetilde{x})=TB_{\frac{1}{2}}$, where $B_{\frac{1}{2}}$
is a ball of radius $\frac{1}{2}$ centered at the origin. By a theorem
of Lelong \cite{Lel57} (see also \cite{Thi67,GL86}), for any analytic
set $M$ of pure dimension $n-m$, and any $r>0$, one has 
\[
\frac{H_{2(n-m)}(M\cap B_{r})}{H_{2(n-m)}(M_{0}\cap B_{r})}\geq\lim_{\rho\to+0}\frac{H_{2(n-m)}(M\cap B_{\rho})}{H_{2(n-m)}(M_{0}\cap B_{\rho})},
\]
where $M_{0}$ is a linear subspace of the same dimension $n-m$.
The limit on the right-hand side is the \emph{Lelong number} of $M$,
which is the algebraic multiplicity of $M$ at $\widetilde{x}$; it
is strictly positive; thus 
\[
H_{2(n-m)}(M\cap B_{r})\gtrsim r^{2(n-m)}.
\]
Applying this estimate to $M=T^{-1}(\widetilde{\phi}^{-1}(\widetilde{\phi}(\widetilde{x})))$
and observing that 
\[
H_{2(n-m)}(TA)\geq\min_{|I|=n-m}\prod_{i\in I}\left|\widetilde{x}_{i}\right|_{\C}\cdot H_{2(n-m)}(A),
\]
for any Borel set $A\subset\widetilde{X}$, we obtain the claimed
assertion. 
\end{proof}
We now proceed with the proof of the theorem. Using Claim \ref{cl:lelong},
we deduce 
\[
\int_{\widetilde{\phi}^{-1}(\widetilde{\phi}(\widetilde{x}))}\frac{\left|\det d_{t}\pi\right|_{\C}\,dH_{2(m-n)}(t)}{G(\pi(t))}\gtrsim\prod_{i=1}^{n}\left|\widetilde{x}_{i}\right|_{\C}^{b_{i}-a_{i}+1},
\]
whence, 
\[
\int_{Y}g(y)^{1+\epsilon}dy\gtrsim\int_{\widetilde{X}}\prod_{i=1}^{n}\left|\widetilde{x}_{i}\right|_{\C}^{b_{i}+\epsilon(b_{i}-a_{i}+1)}d\widetilde{x}.
\]
This integral diverges whenever there is an index $i$ such that 
\[
b_{i}+\epsilon(b_{i}-a_{i}+1)\leq-1,
\]
i.e. 
\[
\epsilon_{\star}(\phi;x)\leq\min_{i}\left\{ \frac{b_{i}+1}{a_{i}-b_{i}-1}\,\,:\,\,b_{i}+1<a_{i}\right\} .
\]
On the other hand,

\[
\operatorname{lct}_{\C}(\mathcal{J}_{\phi};x_{0})=\min_{i}\left\{ \frac{b_{i}+1}{a_{i}}\right\} 
\]
and under the assumption that this quantity is strictly less than
$1$ we have an index $i$ such that 
\[
\operatorname{lct}_{\C}(\mathcal{J}_{\phi};x_{0})=\frac{b_{i}+1}{a_{i}}<1.
\]
For this index, 
\[
\frac{b_{i}+1}{a_{i}-b_{i}-1}=\frac{\operatorname{lct}_{\mathbb{C}}(\mathcal{J}_{\phi};x_{0})}{1-\operatorname{lct}_{\mathbb{C}}(\mathcal{J}_{\phi};x_{0})}.
\]
This concludes the proof of Theorem \ref{thm:upper-C}.

\section{\label{sec:Geometric-characterization-of}Geometric characterization
of $\epsilon_{\star}=\infty$}

From now on let $K$ be a number field and let $\mathcal{O}_{K}$
be its ring of integers. Let $\operatorname{Loc}_{0}$ be the collection
of all non-Archimedean local fields $F$ which contain $K$. We use
the notation $\operatorname{Loc}_{0,\gg}$, for the collection of
$F\in\operatorname{Loc}_{0}$ with large enough residual characteristic,
depending on some given data. Throughout this section, we denote by
$B(y,r)$ a ball of radius $r$ centered at $y\in Y(F)$.

In this section we focus on algebraic morphisms $\varphi:X\rightarrow Y$
between algebraic $K$-varieties. We would like to characterize morphisms
$\varphi$ where $\epsilon_{\star}(\varphi_{F})=\infty$ for certain
collections of local fields $F$, in terms of the singularities of
$\varphi$. In order to effectively do this, it is necessary to consider
an ``algebraically closed'' collection of local fields, such as
the following: 
\begin{itemize}
\item $\{\C\}$. 
\item $\{F\}_{F\in\operatorname{Loc}_{0,\gg}}$, or $\{F\}_{F\in\operatorname{Loc}_{0}}$. 
\end{itemize}
Aizenbud and Avni have shown the following characterization of the
(FRS) property: 
\begin{thm}[{{{\cite[Theorem 3.4]{AA16}}}}]
\label{thm:analytic criterion of the (FRS) property}Let $\varphi:X\rightarrow Y$
be a map between smooth $K$-varieties. Then $\varphi$ is (FRS) if
and only if for each $F\in\operatorname{Loc}_{0}$ and every $\mu\in\mathcal{M}_{c}^{\infty}(X(F))$,
one has $\left(\varphi_{F}\right)_{*}(\mu)\in\mathcal{M}_{c,\infty}(Y(F))$. 
\end{thm}

The Archimedean counterpart of this theorem was studied in \cite{Rei},
where it was shown that given an (FRS) morphism $\varphi$, then $\varphi_{\R}$
and $\varphi_{\C}$ are $L^{\infty}$-morphisms. For the other direction,
the non-Archimedean proof of \cite[Theorem 3.4]{AA16} (see \cite[Section 3.7]{AA16}),
can be easily adapted to the complex case, with less complications
due to the fact that $\C$ is algebraically closed. We arrive at the
following characterization of $L^{\infty}$-morphisms. 
\begin{cor}
\label{cor:characterization of Linfty}Let $\varphi:X\rightarrow Y$
be a map between smooth $K$-varieties. Then the following are equivalent: 
\begin{enumerate}
\item $\varphi$ is (FRS). 
\item For every local field $F$ containing $K$, the map $\varphi_{F}$
is an $L^{\infty}$-morphism. 
\item For each $F\in\operatorname{Loc}_{0,\gg}$, the map $\varphi_{F}$
is an $L^{\infty}$-map. 
\item The map $\varphi_{\C}$ is an $L^{\infty}$-map. 
\end{enumerate}
\end{cor}

Our goal is to characterize the weaker property that $\epsilon_{\star}(\varphi_{F})=\infty$
over $F=\C$, or over all $F\in\operatorname{Loc}_{0}$ (Theorem~\ref{thm:characterization of Lq for all q-intro},
restated below as Theorem~\ref{thm:characterization of Lq for all q}).
Let us first present an example showing the (FRS) condition is too
strong for this purpose. Let $\varphi:\mathbb{A}_{\Q}^{2}\rightarrow\mathbb{A}_{\Q}^{1}$
be the map $\varphi(x,y)=xy$. Then: 
\begin{align*}
\varphi_{\Qp}^{-1}(B(0,p^{-k}))\cap\Z_{p}^{2} & =\left\{ (x,y)\in\Z_{p}^{2}:\val(x)+\val(y)\geq k\right\} \\
 & =\bigcup_{r\geq k,0\leq l\leq r}\left\{ (x,y)\in\Z_{p}^{2}:\val(x)=l,\val(y)=r-l\right\} .
\end{align*}
In particular, we have 
\begin{align*}
\frac{\varphi_{\Qp*}(\mu_{\Z_{p}}^{2})(B(0,p^{-k}))}{\mu_{\mathcal{O}_{F}}(B(0,p^{-k}))} & =p^{k}\cdot\sum_{r=k}^{\infty}\sum_{l=0}^{r}\left(\frac{p-1}{p}\right)^{2}p^{-l}p^{-(r-l)}\\
 & =p^{k}\cdot\left(\frac{p-1}{p}\right)^{2}\sum_{r=k}^{\infty}(r+1)p^{-r}\geq\left(\frac{p-1}{p}\right)^{2}(k+1).
\end{align*}
Hence, the measure $\varphi_{\Qp*}(\mu_{\Z_{p}}^{2})$ does not have
bounded density. On the other hand, since $\operatorname{lct}_{\Qp}(\varphi_{\Qp};0)=1$,
and by considering the asymptotic expansion of $\varphi_{\Qp*}(\mu_{\Z_{p}}^{2})$
as in Theorem \ref{thm:asymptotic expansion}, one sees: 
\begin{enumerate}
\item $\varphi_{\Qp*}(\mu_{\Z_{p}}^{2})$ explodes logarithmically around
$0$, i.e. the density of $\varphi_{\Qp*}(\mu_{\Z_{p}}^{2})$ behaves
like $\val(t)$, around $0$. 
\item $\epsilon_{\star}(\varphi_{\Qp})=\infty$ for every prime $p$. 
\end{enumerate}
By Corollary \ref{cor:characterization of Linfty}, $\varphi$ cannot
be (FRS), and indeed $\left\{ xy=0\right\} $ is not normal, so in
particular it does not have rational singularities.

In order to prove Theorem \ref{thm:characterization of Lq for all q-intro},
we recall the notion of jet schemes. Let $X\subseteq\mathbb{A}_{K}^{n}$
be an affine $K$-scheme whose coordinate ring is 
\[
K[x_{1},\dots,x_{n}]/(f_{1},\dots,f_{k}).
\]
Then the \emph{$m$-th jet scheme} $J_{m}(X)$ of $X$ is the affine
scheme with the following coordinate ring: 
\[
K[x_{1},\dots,x_{n},x_{1}^{(1)},\dots,x_{n}^{(1)},\dots,x_{1}^{(m)},\dots,x_{n}^{(m)}]/(\{f_{j}^{(u)}\}_{j=1,u=1}^{k,m}),
\]
where $f_{i}^{(u)}$ is the $u$-th formal derivative of $f_{i}$.

Let $\varphi:\mathbb{A}^{n_{1}}\rightarrow\mathbb{A}^{n_{2}}$ be
a morphism between affine spaces. Then the \emph{$m$-th jet morphism}
$J_{m}(\varphi):\mathbb{A}^{n_{1}(m+1)}\rightarrow\mathbb{A}^{n_{2}(m+1)}$
of $\varphi$ is given by formally deriving $\varphi$, $J_{m}(\varphi)=(\varphi,\varphi^{(1)},\dots,\varphi^{(m)})$.
Similarly, the $m$-th jet $J_{m}(\varphi)$ of a morphism $\varphi:X\rightarrow Y$
of affine $K$-schemes, is given by the formal derivative of $\varphi$.
Both $J_{m}(X)$ and $J_{m}(\varphi)$ can be generalized to arbitrary
$K$-schemes and $K$-morphisms (see \cite[Chapter 3]{CLNS18} and
\cite{EM09} for more details).

Given a subscheme $Z\subseteq X$ of a smooth variety $X$, with $Z$
defined by an ideal $J$, we denote by $\operatorname{lct}(X,Z):=\operatorname{lct}(J)$
the log-canonical threshold of the pair $(X,Z)$. Musta\c{t}\u{a}
showed that the log-canonical threshold $\operatorname{lct}(X,Z)$
can be characterized in terms of the growth rate of the dimensions
of the jet schemes of $Z$: 
\begin{thm}[{{{\cite[Corollary 0.2]{Mus02}, \cite[Corollary 7.2.4.2]{CLNS18}}}}]
\label{thm:log canonical threshold and jet schemes}Let $X$ be a
smooth, geometrically irreducible $K$-variety, and let $Z\subsetneq X$
be a closed subscheme. Then 
\[
\operatorname{lct}(X,Z)=\dim X-\sup_{m\geq0}\frac{\dim J_{m}(Z)}{m+1}.
\]
Furthermore, the supremum is achieved for $m$ divisible enough. In
particular, $\operatorname{lct}(X,Z)$ is a rational number. 
\end{thm}

Note that $\operatorname{lct}(X,Z)$ depends on $Z$ and $\dim X$,
and neither on the ambient space $X$ not on the embedding of $Z$
in $X$.

We now introduce the following definitions from \cite{GHb}. For a
morphism $\varphi:X\rightarrow Y$ between schemes, we denote by $X_{y,\varphi}$
the scheme theoretic fiber of $\varphi$ over $y\in Y$. 
\begin{defn}
\label{def:jet flat}Let $\varphi:X\rightarrow Y$ be a morphism of
smooth, geometrically irreducible $K$-varieties, and let $\epsilon>0$. 
\begin{enumerate}
\item $\varphi$ is called \textit{$\epsilon$-flat} if for every $x\in X$
we have $\dim X_{\varphi(x),\varphi}\leq\dim X-\epsilon\dim Y$. 
\item $\varphi$ is called \textit{$\epsilon$-jet flat} if $J_{m}(\varphi):J_{m}(X)\rightarrow J_{m}(Y)$
is $\epsilon$-flat for every $m\in\N$. 
\item $\varphi$ is called\textit{ jet-flat} if it is $1$-jet flat. 
\end{enumerate}
In particular, note that $\varphi$ is flat if and only if it is $1$-flat.
\end{defn}

Note that by Theorem \ref{thm:log canonical threshold and jet schemes},
$\varphi$ is $\epsilon$-jet-flat if and only if $\operatorname{lct}(X,X_{\varphi(x),\varphi})\geq\epsilon\dim Y$
for all $x\in X$. We will need the following lemma to give a jet
scheme interpretation to rational and semi-log-canonical singularities
(from Theorem \ref{thm:log canonical threshold and jet schemes}). 
\begin{lem}
\label{lemma: jet flat vs (FRS) and (FSLCS)}Let $\varphi:X\rightarrow Y$
be a morphism of smooth $K$-varieties. Then: 
\begin{enumerate}
\item $\varphi$ is (FRS) if and only if $J_{m}(\varphi)$ is flat with
locally integral fibers, for every $m\geq0$ (in particular, $\varphi$
is jet-flat). 
\item $\varphi$ is jet-flat if and only if $\varphi$ is flat with fibers
of semi-log-canonical singularities. 
\end{enumerate}
\end{lem}

\begin{proof}
Item (1) is proved in \cite[Corollary 3.12]{GHb} and essentially
follows from a characterization of rational singularities, by Musta\c{t}\u{a}
\cite{Mus01}. For the proof of Item (2), note that by \cite[Corollary 2.7]{GHb},
$\varphi$ is jet-flat if and only if $J_{m}(\varphi)$ is flat over
$Y\subseteq J_{m}(Y)$. Since a fiber of a morphism between smooth
varieties is flat if and only if its fibers are local complete intersections,
the latter condition is equivalent to the condition that every $x\in X$,
and every $m\in\N$, the scheme $J_{m}(X_{\varphi(x),\varphi})$ is
a local complete intersection. By \cite[Corollary 10.2.9]{Ish18}
and \cite[Corollary 3.17]{EI15}, this is equivalent to the condition
that $X_{\varphi(x),\varphi}$ has semi-log-canonical singularities,
for every $x\in X$. 
\end{proof}

\subsection{Proof of Theorem \ref{thm:characterization of Lq for all q-intro}}

We are now in a position to prove Theorem \ref{thm:characterization of Lq for all q-intro}.
Let us recall its formulation, slightly restated using Lemma \ref{lemma: jet flat vs (FRS) and (FSLCS)}. 
\begin{thm}
\label{thm:characterization of Lq for all q}Let $\varphi:X\rightarrow Y$
be a map between smooth $K$-varieties. Then the following are equivalent: 
\begin{enumerate}
\item $\varphi$ is jet-flat. 
\item For every local field $F$ containing $K$, we have $\epsilon_{\star}(\varphi_{F})=\infty$. 
\item For every $F\in\operatorname{Loc}_{0,\gg}$ we have $\epsilon_{\star}(\varphi_{F})=\infty$. 
\item We have $\epsilon_{\star}(\varphi_{\C})=\infty$. 
\end{enumerate}
\end{thm}

The proof of Theorem \ref{thm:characterization of Lq for all q} is
done by showing both implications $(1)\Rightarrow(2)\Rightarrow(3)\Rightarrow(1)$
and $(1)\Rightarrow(2)\Rightarrow(4)\Rightarrow(1)$. The implications
$(2)\Rightarrow(3)$ and $(2)\Rightarrow(4)$ are immediate. We first
prove $(3)\Rightarrow(1)$ and $(4)\Rightarrow(1)$. Then we will
prove $(1)\Rightarrow(2)$ in the non-Archimedean case in $\mathsection$\ref{subsec:--non-Archimedean-case},
and the Archimedean case in $\mathsection$\ref{subsec:-Archimedean-case}. 
\begin{prop}
\label{prop:(3),(4)-->(1)}Let $\varphi:X\rightarrow Y$ be a map
between smooth $K$-varieties. Assume that either $\epsilon_{\star}(\varphi_{\C})=\infty$
or $\epsilon_{\star}(\varphi_{F})=\infty$ for all $F\in\operatorname{Loc}_{0,\gg}$.
Then $\varphi$ is jet-flat. 
\end{prop}

\begin{proof}
Working locally, and composing $\varphi$ with an \'etale map $\Phi:Y\rightarrow\mathbb{A}^{m}$,
we may assume that $Y=\mathbb{A}^{m}$. Let $\psi:\mathbb{A}^{m}\rightarrow\mathbb{A}^{m}$
be any dominant morphism, and let $\mu_{1}\in\mathcal{M}_{c}^{\infty}(X(F))$
and $\mu_{2}\in\mathcal{M}_{c}^{\infty}(F^{m})$. By Theorem \ref{thm:lowerbd},
for all $F\in\operatorname{Loc}_{\gg}\cup\left\{ \C\right\} $, we
have $\psi_{*}(\mu_{2})\in L^{1+\epsilon}$ for some $\epsilon>0$.
Taking $q$ large enough, by Young's convolution inequality, one has:
\[
(\varphi*\psi)_{*}(\mu_{1}\times\mu_{2})=\varphi_{*}(\mu_{1})*\psi_{*}(\mu_{2})\subseteq L^{q}*L^{1+\epsilon}\subseteq L^{\infty},
\]
where $\varphi*\psi$ is as in Definition \ref{def:convolution}.
By Corollary \ref{cor:characterization of Linfty}, we get that $\varphi*\psi$
is (FRS). We now claim that since $\varphi$ is a morphism whose convolution
with any dominant morphism produces an (FRS) morphism, $\varphi$
must be jet-flat.

Indeed, assume it is not the case. Then by Theorem \ref{thm:log canonical threshold and jet schemes}
and Definition \ref{def:jet flat}, there exist $y\in\overline{K}^{m}$
and $N>0$ such that the scheme theoretic fiber $X_{y,\varphi}$ of
$\varphi$ over $y$ satisfies 
\[
\operatorname{lct}(X,X_{y,\varphi})=\dim X-\sup_{k\geq0}\frac{\dim J_{k}X_{y,\varphi}}{k+1}\leq m(1-\frac{2}{N}).
\]
Moreover, this supremum is achieved for $k$ divisible enough. Thus
the map 
\[
J_{k}(\varphi):J_{k}(X)\rightarrow J_{k}(Y)
\]
is not $\left(1-\frac{1}{N}\right)$-flat for $k$ divisible enough.
But on the other hand, the map 
\[
\psi_{N}(y_{1},...,y_{m})=(y_{1}^{2N},...,y_{m}^{2N})
\]
satisfies that $J_{l}(\psi_{N})$ is not $\frac{1}{N}$-flat for divisible
enough $l$ (since $\operatorname{lct}(y_{i}^{2N})=\frac{1}{2N}$).
Thus we may find $k_{0}\in\N$ such that $J_{k_{0}}(\varphi)$ is
not $\left(1-\frac{1}{N}\right)$-flat and $J_{k_{0}}(\psi_{N})$
is not $\frac{1}{N}$-flat . But then $J_{k_{0}}(\varphi*\psi)=J_{k_{0}}(\varphi)*J_{k_{0}}(\psi)$
is not flat (see \cite[Lemma 3.26]{GHb}), which is a contradiction
by Fact \ref{lemma: jet flat vs (FRS) and (FSLCS)}, as $\varphi*\psi$
is (FRS). 
\end{proof}

\subsection{\label{subsec:--non-Archimedean-case}$(1)\Rightarrow(2)$: the non-Archimedean
case }

We now turn to the proof of $(1)\Rightarrow(2)$, in the non-Archimedean
case. We first prove the following variant of \cite[Theorem 4.12]{CGH23}. 
\begin{prop}[{\cite[Theorem 4.12]{CGH23}}]
\label{Proposition:logarithmic explosion}Let $\varphi:X\rightarrow Y$
be a jet-flat map between smooth $K$-varieties. Then there exists
$M\in\N$, such that for each $F\in\operatorname{Loc}_{0}$, each
$\mu\in\mathcal{M}_{c}^{\infty}(X(F))$ and each non-vanishing $\tau\in\mathcal{M}^{\infty}(Y(F))$,
one can find $C_{F,\mu,\tau}>0$ such that for each $y\in Y(F)$ and
$k\in\N$ one has, 
\[
G_{F,\mu}(y,k):=\frac{(\varphi_{*}\mu)(B(y,q_{F}^{-k}))}{\tau(B(y,q_{F}^{-k}))}\leq C_{F,\mu,\tau}k^{M}.
\]
\end{prop}

\begin{proof}
We may assume that $Y=\mathbb{A}_{K}^{m}$ and that $\tau=\mu_{F}^{m}$.
We may further assume that $X$ is affine, and thus embeds in $\mathbb{A}_{K}^{n}$.
Let $\widetilde{\mu}$ be the canonical measure on $X(F)$ (see \cite[Section 3.3]{Ser81},
and also \cite[Section 1.2]{CCL12}). It is enough to consider measures
$\mu\in\mathcal{M}_{c}^{\infty}(X(F))$ which are of the form $\mu_{l}:=\widetilde{\mu}|_{B(0,q_{F}^{l})\cap X(F)}$.

Write $g_{F}(l,y)$ for the density of $\varphi_{*}\mu_{l}$ with
respect to $\mu_{F}^{m}$, and set $G_{F}(y,l,k):=G_{F,\mu_{l}}(y,k)$.
Then the collections $\{g_{F}(l,y)\}_{F\in\operatorname{Loc}_{0}}$
and $\{G_{F}(y,l,k)\}_{F\in\operatorname{Loc}_{0}}$ are both motivic
functions, in the sense of \cite[Section 1.2]{CGH18}. By \cite[Theorem 2.1.3]{CGH18},
there exists a motivic function $H(l,k)=\left\{ H_{F}:\Z\times\N\rightarrow\R\right\} _{F\in\operatorname{Loc}_{0}}$,
which approximates the supremum of $G(y,l,k)$, that is: 
\begin{equation}
\frac{1}{C_{F}}H_{F}(l,k)\leq\sup_{y\in Y(F)}G_{F}(y,l,k)\leq H_{F}(l,k),\label{eq:approximate of suprema}
\end{equation}
for all $F\in\operatorname{Loc}_{0}$ and $(l,k)\in\Z\times\N$, where
$C_{F}$ depends only on the local field\footnote{In the statement of \cite[Theorem 2.1.3]{CGH18}, the approximation
(\ref{eq:approximate of suprema}) is stated for $\left|G_{F}(y,l,k)\right|_{\C}$
instead of $G_{F}(y,l,k)$. Since $G_{F}(y,l,k)$ is a non-negative
real-valued motivic function, their argument yields the current statement
as well (see the first four lines of the proof on p.146).}.

Since $H$ is motivic, and using \cite[Proposition 1.4.2]{CGH18},
for each $F\in\operatorname{Loc}_{0}$, we may divide $\Z\times\N$
into a finite disjoint union $\Z\times\N=\bigsqcup_{A_{F}\in\mathcal{A}}A_{F}$,
with $\left|\mathcal{A}\right|<\infty$ independent of $F$, such
that on each part $A_{F}\subseteq\Z\times\N$, the following hold.
There exist finitely many $a_{i}\in\N$, $b_{i}\in\Q$ independent
of $F$, a finite set $\Lambda_{F}$ (of size depending on $F$),
and a finite partition of $A_{F}$ into subsets $A_{F}=\bigsqcup_{\xi\in\Lambda_{F}}A_{F,\xi}$,
such that for all $(k,l)\in A_{F,\xi}$: 
\[
H_{F}(l,k)=\sum_{i=1}^{L}c_{i}(\xi,l,F)\cdot k^{a_{i}}q_{F}^{b_{i}k},
\]
for some constants $c_{i}(\xi,l,F)$ depending on $l,F$ and $\xi$.
Moreover, for fixed $F\in\operatorname{Loc}_{0}$ and $l\in\N$, the
set $A_{F,\xi,l}:=\left\{ k\in\N:(k,l)\in A_{F,\xi}\right\} $ is
either finite or a fixed congruence class modulo some $e\in\Z_{\geq1}$.

To prove the proposition, it is enough to show that for each $F\in\operatorname{Loc}_{0}$
and $l\in\Z$, we have $H_{F}(l,k)\leq C_{F,l}k^{M}$ on each $A_{F,\xi,l}$,
for some constant $C_{F,l}$ depending on $F,l$. It is enough to
prove this for $A_{F,\xi,l}$ infinite, as otherwise we have 
\[
H_{F}(l,k)\leq C_{F,l}:=\sum_{k\in A_{F,\xi,l}}\sum_{i=1}^{L}\left|c_{i}(\xi,l,F)\right|\cdot k^{a_{i}}q_{F}^{b_{i}k}.
\]
Now suppose $A_{F,\xi,l}$ is infinite. By rearranging the constants
$c_{i}(\xi,l,F)$, we may assume that the pairs $(a_{i},b_{i})$ are
disjoint, and that $(b_{i},a_{i})>(b_{i+1},a_{i+1})$ in lexicographic
order, that is, either $b_{i}>b_{i+1}$ or $b_{i}=b_{i+1}$ and $a_{i}>a_{i+1}$.
Note that if $b_{1}\leq0$, then we are done, since for each $F\in\operatorname{Loc}_{0}$
and each $k\in A_{F,\xi,l}$: 
\begin{equation}
\sup_{y\in Y(F)}G_{F}(y,l,k)\leq H_{F}(l,k)\leq\left(\sum_{i=1}^{L}\left|c_{i}(\xi,l,F)\right|\right)k^{M},\label{eq:bound on G(y,l,k)}
\end{equation}
where $M=\max\{a_{i}\}$. Assume towards contradiction that $b_{1}>0$,
and $c_{1}(\xi,l,F)\neq0$. Then for all large enough $k\in A_{F,\xi,l}$,
one has 
\begin{equation}
\sup_{y\in Y(F)}G_{F}(y,l,k)\geq C_{F}^{-1}H_{F}(l,k)\geq q_{F}^{\frac{1}{2}b_{1}k}.\label{eq:lower bound on G(y,k)}
\end{equation}
Now let $\psi_{R}:\mathbb{A}_{K}^{m}\rightarrow\mathbb{A}_{K}^{m}$
be the map $\psi_{R}(x_{1},...,x_{m})=(x_{1}^{R},...,x_{m}^{R})$,
for $R:=\left\lceil 4m/b_{1}\right\rceil $. Then by \cite[Corollary 3.18]{GHb},
$\varphi*\psi_{R}:X\times\mathbb{A}_{K}^{m}\rightarrow\mathbb{A}_{K}^{m}$
is (FRS). By Corollary \ref{cor:characterization of Linfty}, we have
\begin{equation}
(\varphi*\psi_{R})_{*}(\mu_{l}\times\mu_{\mathcal{O}_{F}}^{m})\in\mathcal{M}_{c,\infty}(F^{m}).\label{eq:(FRS)}
\end{equation}
On the other hand, note that 
\[
(\varphi*\psi_{R})^{-1}(B(y,q_{F}^{-k}))\supseteq\varphi^{-1}(B(y,q_{F}^{-k}))\times\psi_{R}^{-1}(B(0,q_{F}^{-k})).
\]
Further note that 
\[
\psi_{R}^{-1}(B(0,q_{F}^{-k}))=B\left(0,q_{F}^{-\left\lceil \frac{k}{R}\right\rceil }\right).
\]
Thus, we have 
\begin{align*}
\frac{(\varphi*\psi_{R})_{*}(\mu_{l}\times\mu_{\mathcal{O}_{F}}^{m})(B(y,q_{F}^{-k}))}{\mu_{F}^{m}(B(y,q_{F}^{-k}))} & =\frac{\mu_{l}\times\mu_{\mathcal{O}_{F}}^{m}((\varphi*\psi_{R})^{-1}(B(y,q_{F}^{-k}))}{\mu_{F}^{m}(B(y,q_{F}^{-k}))}\\
 & \geq\frac{\mu_{l}\left(\varphi^{-1}(B(y,q_{F}^{-k}))\right)}{\mu_{F}^{m}(B(y,q_{F}^{-k}))}\cdot\mu_{\mathcal{O}_{F}}^{m}\left(\psi_{R}^{-1}(B(0,q_{F}^{-k}))\right)\\
 & =\frac{\varphi_{*}\mu_{l}(B(y,q_{F}^{-k}))}{\mu_{F}^{m}(B(y,q_{F}^{-k}))}\cdot\mu_{\mathcal{O}_{F}}^{m}\left(B\left(0,q_{F}^{-\left\lceil \frac{k}{R}\right\rceil }\right)\right)\\
 & \geq G_{F}(y,l,k)\cdot q_{F}^{-\left\lceil \frac{k}{R}\right\rceil m}.
\end{align*}
By (\ref{eq:lower bound on G(y,k)}), for each $k$ large enough,
we may find $y_{0}\in Y(F)$, such that 
\[
\frac{(\varphi*\psi_{R})_{*}(\mu_{l}\times\mu_{\mathcal{O}_{F}}^{m})(B(y_{0},q_{F}^{-k}))}{\mu_{F}^{m}(B(y_{0},q_{F}^{-k}))}\geq q_{F}^{-\left\lceil \frac{k}{R}\right\rceil m}q_{F}^{\frac{1}{2}b_{1}k}\geq q_{F}^{\frac{b_{1}}{8}k},
\]
which contradicts (\ref{eq:(FRS)}). Thus $b_{1}\leq0$ and we are
done by (\ref{eq:bound on G(y,l,k)}). 
\end{proof}
We are now ready to prove $(1)\Rightarrow(2)$ of Theorem \ref{thm:characterization of Lq for all q}. 
\begin{proof}[Proof of $(1)\Rightarrow(2)$ of Theorem \ref{thm:characterization of Lq for all q},
non-Archimedean case]
Let $\varphi:X\to Y$ be a jet-flat morphism. We may assume $Y=\mathbb{A}_{K}^{m}$.
Let $\mu\in\mathcal{M}_{c}^{\infty}(X(F))$ and write $g_{F}$ for
the density of $\varphi_{*}\mu$ with respect to $\mu_{F}^{m}$. Let
$Y^{\mathrm{sm},\varphi}$ be the set of $y\in Y$ such that $\varphi$
is smooth over $y$. For every $F\in\operatorname{Loc}_{0}$, the
map $\varphi_{F}$ is smooth over $Y^{\mathrm{sm},\varphi}(F)$, and
therefore $g_{F}(y)$ is locally constant on $Y^{\mathrm{sm},\varphi}(F)$.

By \cite[Corollary 1.4.3]{CGH18}, the constancy radius of $g_{F}(y)$
can be taken to be definable, i.e. there exists a definable function
$\alpha:Y^{\mathrm{sm},\varphi}\to\N$ such that $g_{F}(y)$ is constant
around every ball $B(y,q_{F}^{-\alpha_{F}(y)})$. In particular, for
every $y\in Y^{\mathrm{sm},\varphi}(F)$ we have $g_{F}(y)=G_{F,\mu}(y,\alpha_{F}(y))$.
In addition, by Proposition \ref{Proposition:logarithmic explosion},
we have $G_{F,\mu}(y,k)\leq C_{F,\mu}k^{M}$, for $F\in\operatorname{Loc}_{0}$.
We arrive at the following: 
\begin{align*}
\int_{F^{m}}\left|g_{F}(y)\right|^{s}dy & =\int_{Y^{\mathrm{sm},\varphi}(F)}\left|G_{F,\mu}(y,\alpha_{F}(y))\right|^{s}dy\leq C_{F,\mu}^{s}\int_{F^{m}}\alpha_{F}(y)^{Ms}dy\\
 & =C_{F,\mu}^{s}\sum_{t\in\N}t^{Ms}\cdot\mu_{F}^{m}(\{y\in F^{m}:\alpha_{F}(y)=t\})\\
 & \leq C'(F)+C\sum_{t\in\N}t^{Ms}q_{F}^{-\lambda t}<\infty,
\end{align*}
where the last inequality follows by \cite[Theorem 3.1.1]{CGH18},
since $\lim\limits _{t\to\infty}\mu_{F}^{m}(\{y\in F^{m}:\alpha_{F}(y)=t\})=0$
and thus $\mu_{F}^{m}(\{y\in F^{m}:\alpha_{F}(y)=t\})\leq q_{F}^{-\lambda t}$
for some $\lambda>0$ and every $t$ large enough. 
\end{proof}
In \cite[Theorem 4.12]{CGH23}, Cluckers and the first two authors
showed that if $\varphi:X\to Y$ is a jet-flat morphism, which is
defined over $\Z$, and one chooses $\mu=\mu_{X(\Zp)}$ and $\tau=\mu_{Y(\Zp)}$
to be the canonical measures on $X(\Zp)$ and $Y(\Zp)$ (see \cite[Lemma 4.2]{CGH23}),
then the constant $C_{\Qp,\mu,\tau}$ in Proposition \ref{Proposition:logarithmic explosion}
can be taken to be independent of $p$ (i.e. $C_{\Qp,\mu,\tau}=C$).
\cite[Theorem 4.12]{CGH23}, together with the ideas of the proof
of $(1)\Rightarrow(2)$ of Theorem \ref{thm:characterization of Lq for all q},
allows us to give bounds on the $L^{s}$ norms of $\frac{\varphi_{*}\mu_{X(\Zp)}}{\mu_{Y(\Zp)}}$,
which are independent of $p$: 
\begin{prop}
Let $\varphi:X\to Y$ be a dominant morphism between finite type $\Z$-schemes
$X$ and $Y$, with $X_{\Q},Y_{\Q}$ smooth and geometrically irreducible.
For any prime $p$, let $g_{p}$ be the density of $\varphi_{*}\mu_{X(\Zp)}$
with respect to $\mu_{Y(\Zp)}$. Then the following are equivalent: 
\begin{enumerate}
\item $\varphi_{\Q}:X_{\Q}\to Y_{\Q}$ is jet-flat. 
\item For every $s>1$, there exists $C(s)>0$ such that for every prime
$p$,
\[
\left\Vert g_{p}\right\Vert _{s}:=\int_{Y(\Zp)}\left|g_{p}(y)\right|^{s}d\mu_{Y(\Zp)}<C(s).
\]
\end{enumerate}
\end{prop}

\begin{proof}
The implication $(2)\Rightarrow(1)$ follows from Proposition \ref{prop:(3),(4)-->(1)}.
By $(1)\Rightarrow(2)$ of Theorem \ref{thm:characterization of Lq for all q},
it is enough to prove $(1)\Rightarrow(2)$ for $p\gg_{s}1$. Suppose
that $\varphi_{\Q}$ is jet-flat. We may assume that $Y=\mathbb{A}^{m}$.
Indeed, working locally, and since $Y_{\Q}$ is smooth, we may assume
there exists a morphism $\psi:Y\rightarrow\mathbb{A}^{m}$, such that
$\psi_{\Q}$ is an \'etale map. If $\widetilde{g}_{p}$ is the density
of $\psi_{*}\varphi_{*}\mu_{X(\Zp)}$ with respect to $\mu_{p}^{m}$,
and $N$ is an upper bound on the size of the geometric fibers of
$\psi$, then
\[
\left|\frac{\psi_{*}\mu_{Y(\Zp)}}{\mu_{p}^{m}}\right|<N,
\]
for $p$ large enough. In particular, we have:
\[
\int_{Y(\Zp)}\left|g_{p}(y)\right|^{s}d\mu_{Y(\Zp)}\leq\int_{Y(\Zp)}\left|\widetilde{g}_{p}\circ\psi(y))\right|^{s}d\mu_{Y(\Zp)}=\int_{\Z_{p}^{m}}\left|\widetilde{g}_{p}(\widetilde{y})\right|^{s}d(\psi_{*}\mu_{Y(\Zp)})\leq N\int_{\Z_{p}^{m}}\left|\widetilde{g}_{p}(\widetilde{y})\right|^{s}d\mu_{p}^{m}.
\]
As in the proof of $(1)\Rightarrow(2)$ of Theorem \ref{thm:characterization of Lq for all q}
above, by \cite[Corollary 1.4.3]{CGH18}, there exists a definable
function $\alpha:(\mathbb{A}^{m})^{\mathrm{sm},\varphi}\to\N$ such
that $g_{p}(y)=G_{p}(y,\alpha_{\Qp}(y))$, where $G_{p}(y,k):=G_{\Qp,\mu_{X(\Zp)}}(y,k)$.
Hence, applying \cite[Theorem 4.12]{CGH23} we can find $C,M\in\N$
such that for $p\gg1$:
\begin{align*}
\int_{\Z_{p}^{m}}\left|g_{p}(y)\right|^{s}dy & =\int_{(\mathbb{A}^{m})^{\mathrm{sm},\varphi}(\Qp)\cap\Z_{p}^{m}}\left|G_{p}(y,\alpha_{\Qp}(y))\right|^{s}dy\leq C^{s}\int_{(\mathbb{A}^{m})^{\mathrm{sm},\varphi}(\Qp)\cap\Z_{p}^{m}}\alpha_{\Qp}(y)^{Ms}dy\\
 & =C^{s}\sum_{t\in\N}t^{Ms}\cdot\mu_{p}^{m}(\{y\in\Z_{p}^{m}:\alpha_{\Qp}(y)=t\}).
\end{align*}
By \cite[Theorem 3.1.1]{CGH18}, there exists $L\in\N$ and $\lambda>0$
such that $\mu_{p}^{m}(\{y\in\Z_{p}^{m}:\alpha_{\Qp}(y)=t\})\leq p^{-\lambda t}$
for every $t>L$ and every prime $p$. We therefore get the desired
claim as:
\[
\int_{\Z_{p}^{m}}\left|g_{p}(y)\right|^{s}dy\leq C^{s}\sum_{t=0}^{L}t^{Ms}+C^{s}\sum_{t>L}t^{Ms}p^{-\lambda t}<C^{s}\left((L+1)\cdot L^{Ms}+\sum_{t>L}t^{Ms}2^{-\lambda t}\right)<C(s).\qedhere
\]
\end{proof}

\subsection{\label{subsec:-Archimedean-case}$(1)\Rightarrow(2)$: the Archimedean
case }

In this subsection we prove $(1)\Rightarrow(2)$ in the cases $F=\R$
and $F=\C$. Let $\varphi:X\rightarrow Y$ be a jet-flat morphism
between smooth algebraic varieties, defined over $F$. Using restriction
of scalars, we may assume that $F=\R$. We would like to show that
for each $\mu\in\mathcal{M}_{c}^{\infty}(X(\R))$, we have $\varphi_{*}\mu\in\mathcal{M}^{q}(Y(\R))$
for all $1\leq q<\infty$. We start with the following proposition. 
\begin{prop}
\label{Prop:logarithmic explosion-Archimedean}Let $\varphi:X\rightarrow Y$
be a jet-flat map between smooth $\R$-varieties. Then for every $\mu\in\mathcal{M}_{c}^{\infty}(X(\R))$,
every non-vanishing $\tau\in\mathcal{M}^{\infty}(Y(\R))$, and every
$p\in\N$, one can find $C_{\mu,\tau,p}>0$ and $M_{\mu,p}>0$ such
that for each $y\in Y(\R)$ and $0<r<\frac{1}{2}$ one has 
\[
\int_{Y(\R)}\left(\frac{\left(\varphi_{*}\mu\right)(B(y,r))}{\tau(B(y,r))}\right)^{p}d\tau(y)<C_{\mu,\tau,p}\left|\log(r)\right|^{M_{\mu,p}}.
\]
\end{prop}

\begin{rem}
Proposition \ref{Prop:logarithmic explosion-Archimedean} is weaker
than its non-Archimedean counterpart (Proposition \ref{Proposition:logarithmic explosion}).
The main obstacle is that we do not know of an Archimedean analogue
to \cite[Theorem 2.1.3]{CGH18}, that is, whether one can approximate
the supremum of constructible functions by constructible functions
as in (\ref{eq:approximate of suprema}). It is conjectured by Raf
Cluckers that such a statement should be true under suitable assumptions
(see \cite[Conjecture 6.9]{AM}). We thank the anonymous referee for
bringing this to our attention. We further believe that the current
proposition should hold for $M_{\mu,p}=C\cdot p$ for a sufficiently
large $C$ independent of $\mu$. 
\end{rem}

Analogously to the non-Archimedean case, we introduce the following
notion of constructible functions: 
\begin{defn}[{\cite[Section 1.1]{CM11}, see also \cite{LR97}}]
\label{def:constructible functions}~ 
\begin{enumerate}
\item A \emph{restricted analytic function} is a function $f:\R^{n}\rightarrow\R$
such that $f|_{[-1,1]^{n}}$ is analytic and $f|_{\R^{n}\backslash[-1,1]^{n}}=0$. 
\item A subset $A\subseteq\R^{n}$ is \emph{subanalytic} if it is definable
in $\mathbb{R}_{\mathrm{an}}$ \textendash{} the extension of the
ordered real field by all restricted analytic functions. A function
$f:A\rightarrow B$ is \emph{subanalytic} if its graph $\Gamma_{f}\subseteq\R^{n}\times\R^{m}$
is subanalytic. 
\item A function $h:A\rightarrow\R$ is called \emph{constructible} if there
exist subanalytic functions $f_{i}:A\rightarrow\R$ and $f_{ij}:A\rightarrow\R_{>0}$,
such that: 
\begin{equation}
h(x)=\sum_{i=1}^{N}f_{i}(x)\cdot\prod_{j=1}^{N_{i}}\log(f_{ij}(x)).\label{eq:real constructible}
\end{equation}
We denote the class of constructible functions on $A$ by $\mathcal{C}(A)$. 
\item Given an analytic manifold $Z$, a measure $\mu\in\mathcal{M}(Z)$
is called \emph{constructible} if locally it is of the form $f\cdot\left|\omega\right|$,
where $\omega$ is a regular top-form, and $f$ is constructible.
We denote the class of constructible measures by $\mathcal{CM}(Z)$.
Similarly, we write $\mathcal{CM}_{c,q}(Z)$, $\mathcal{CM}^{\infty}(Z)$
and $\mathcal{CM}_{c}^{\infty}(Z)$. 
\end{enumerate}
\end{defn}

Note that $X(\R)$ is defined by polynomials, so it is definable in
the real field, and in particular subanalytic. Since $\mathcal{C}(X(\R))$
contains indicators of balls, we may assume that $\mu\in\mathcal{CM}_{c,\infty}(X(\R))$
when proving Proposition \ref{Prop:logarithmic explosion-Archimedean}
and Theorem \ref{thm:characterization of Lq for all q}. 
\begin{proof}[Proof of Proposition \ref{Prop:logarithmic explosion-Archimedean}]
We may assume that $Y=\mathbb{A}_{\R}^{m}$, $\tau=\mu_{\R}^{m}$
and $\mu\in\mathcal{CM}_{c,\infty}(X(\R))$. Write $g(y)$ for the
density of $\varphi_{*}\mu$ with respect to $\mu_{\R}^{m}$. Set
\begin{equation}
G(y,r)=\frac{(\phi_{*}\mu)(B(y,r))}{r^{m}}=\frac{1}{r^{m}}\varint_{B(y,r)}g(y')dy'.\label{eq:avaraging of g}
\end{equation}
For each $p\in\N$, let 
\[
G_{p}(r):=\int_{\R^{m}}G(y,r)^{p}dy.
\]
By \cite[Theorem 1.3]{CM11}, the functions $G:\R^{m}\times\R_{>0}\rightarrow\R$
and $G_{p}(r):\R_{>0}\rightarrow\R$ are constructible. Writing $G_{p}$
as in (\ref{eq:real constructible}), and using a preparation theorem
for constructible functions \cite[Corollary 3.5]{CM12}, there exist
$\delta>0$ and $\theta\in\R$, such that $\theta=0$ or $\theta\notin[0,\delta]$,
and for each $r\in(0,\delta)$ one can write: 
\begin{equation}
G_{p}(r)=\sum_{i=1}^{M}d_{i}\cdot S_{i}(r)\cdot\left|r-\theta\right|^{\alpha_{i}}\log(\left|r-\theta\right|)^{l_{i}},\label{eq:constructible prepared}
\end{equation}
where $d_{i}\in\R$, $l_{i},M\in\N$, $\alpha_{i}\in\Q$, and where
$S_{i}$ are certain subanalytic units, called \emph{strong functions}
(see \cite[Definition 2.3]{CM11}). We may assume that $\theta=0$
as otherwise, $G_{p}(r)$ is bounded on $(0,\delta)$ and we are done.
In addition, \cite[Corollary 3.5]{CM12} also ensures that for each
$1\leq i\leq M$, either $S_{i}(r)=1$ for each $r\in(0,\delta)$,
or $\alpha_{i}>-1$. This additional property is achieved by writing
each strong function $S_{i}(r)$ as a converging infinite sum $S_{i}(r)=\sum_{j=0}^{\infty}c_{ij}\cdot r^{j/p}$
for some $p\in\Q_{\geq0}$, and split it into a finite sum $\sum_{j=0}^{s}c_{ij}\cdot r^{j/p}$
and an infinite sum $\widetilde{S_{i}}(r):=\sum_{j>s}c_{ij}\cdot r^{j/p}$.
By taking $s$ large enough, and rearranging the terms in (\ref{eq:constructible prepared}),
Cluckers and Miller ensured that $\alpha_{i}>-1$ whenever $S_{i}(r)\neq1$.
Following the same argument, and taking $s$ even larger, one can
guarantee that $\alpha_{i}>N$ for some fixed $N\in\N$, as large
as we wish. Hence, we may assume that $G_{p}(r)$ has the following
form: 
\begin{equation}
G_{p}(r)=\sum_{i=1}^{M'}d_{i}r^{\alpha_{i}}\log(r)^{l_{i}}+\sum_{i=M'+1}^{M}d_{i}S_{i}(r)r^{\alpha_{i}}\log(r)^{l_{i}},\label{eq:rewriting G_p(r)}
\end{equation}
where $\alpha_{i}>N$ for $M'+1\leq i\leq M$, and $N\in\N$ large
as we like. For $1\leq i\leq M'$, we may further assume that $(\alpha_{i},l_{i})$
are mutually different and lexicographically ordered, i.e.\ either
$\alpha_{i}<\alpha_{i+1}$, or $\alpha_{i}=\alpha_{i+1}$ and $l_{i}>l_{i+1}$.
In particular, by taking $\delta$ small enough, we have for $0<r<\delta$:
\begin{equation}
\frac{1}{2}d_{1}r^{\alpha_{1}}\left|\log(r)\right|^{l_{1}}<\left|G_{p}(r)\right|<2d_{1}r^{\alpha_{1}}\left|\log(r)\right|^{l_{1}}.\label{eq:assymptotic expansion of constructible functions}
\end{equation}
We claim that $\alpha_{1}\geq0$. Assume not, then we have for $r$
small enough: 
\begin{equation}
\left\Vert G(\cdot\,,r)\right\Vert _{\infty}\geq\left\Vert G(\cdot\,,r)\right\Vert _{p}=G_{p}(r)^{\frac{1}{p}}\gtrsim r^{\frac{\alpha_{1}}{p}}\left|\log(r)\right|^{\frac{l_{1}}{p}}\gtrsim r^{\frac{\alpha_{1}}{p}}.\label{eq:lower bound on infty norm}
\end{equation}
We now use an argument analogous to the one in Proposition \ref{Proposition:logarithmic explosion}.
Take $\psi_{R}:\mathbb{A}^{m}\rightarrow\mathbb{A}^{m}$ to be the
map 
\[
\psi_{R}(x_{1},...,x_{m})=(x_{1}^{R},...,x_{m}^{R}),
\]
for $R:=\left\lceil 2mp/\left|\alpha_{1}\right|\right\rceil $ and
let $\eta\in C_{c}^{\infty}(\R^{m})$ be a bump function which is
equal to one on the unit ball in $\R^{m}$. Then $\varphi*\psi_{R}:X\times\mathbb{A}^{m}\rightarrow\mathbb{A}^{m}$
is (FRS), and thus by Corollary \ref{cor:characterization of Linfty}
\begin{equation}
(\varphi*\psi_{R})_{*}(\mu\times\eta)\in\mathcal{M}_{c,\infty}(\R^{m}).\label{eq:(FRS)-Archimedean}
\end{equation}
Repeating precisely the same argument as in Proposition \ref{Proposition:logarithmic explosion},
and using (\ref{eq:lower bound on infty norm}), we may find $\{y_{r}\}_{r}$
such that: 
\[
\frac{(\varphi*\psi_{R})_{*}(\mu\times\eta)(B(y_{r},2r))}{\mu_{\R}^{m}(B(y_{r},2r))}\gtrsim G(y_{r},r)\cdot\psi_{R*}\eta(B(0,r))\gtrsim G(y_{r},r)r^{\frac{m}{R}}\gtrsim r^{\frac{\alpha_{1}}{4p}},
\]
which leads to a contradiction. Hence $\alpha_{1}\geq0$, therefore
on $(0,\delta)$ for $\delta$ small enough we have 
\[
G_{p}(r)\lesssim\left|\log(r)\right|^{l_{1}}.
\]
For $r>\delta$, we have $G(y,r)\lesssim\delta^{-m}$ and thus $G_{p}(r)\lesssim\delta^{-pm}<\infty$.
This concludes the proof. 
\end{proof}
In order to prove Theorem \ref{thm:characterization of Lq for all q},
we need to control the oscillations of constructible functions. 
\begin{defn}
\label{def:local constancy real}Let $f:\R^{n}\rightarrow\R$ be a
subanalytic function. Define $\alpha_{f}:\R^{n}\times\R_{>0}\rightarrow\R_{\geq0}$
by 
\begin{equation}
\alpha_{f}(y,r):=\min\left(1,\sup\left\{ t\in\R_{\geq0}:\forall y'\in B(y,t),\,\left|f(y)-f(y')\right|<r\right\} \right).\label{eq:radius of approx constancy}
\end{equation}
\end{defn}

Note that $\alpha_{f}$ is subanalytic, and for any $r>0$ we have
$\alpha_{f}(y,r)>0$ for almost every $y$. The next lemma extends
this construction to the ring of constructible functions. 
\begin{lem}
\label{Lemma:equals avarage}Let $g\in\mathcal{C}(\R^{n})$. Then
there exists a subanalytic function $\alpha_{g}:\R^{n}\times\R_{>0}\rightarrow\R_{\geq0}$
such that for any $r>0$ we have $\alpha_{g}(y,r)>0$ for almost all
$y\in\R^{n}$, and 
\begin{equation}
\left|g(y)-g(y')\right|\leq r\text{ for all }y'\in B(y,\alpha_{g}(y,r)).\label{eq:approx on balls}
\end{equation}
\end{lem}

\begin{proof}
If $g=\sum_{i=1}^{M}g_{i}$ for $g_{i}\in\mathcal{C}(\R^{n})$, and
suppose we already constructed $\alpha_{g_{i}}$, for each $i$. Then
we may set $\alpha_{g}(y,r)=\min_{i}\,\alpha_{g_{i}}(y,\frac{r}{M})$.
Hence, by (\ref{eq:real constructible}), we may assume that 
\begin{equation}
g(y)=f(y)\cdot\prod_{j=1}^{N}\log(f_{j}(y)),\label{eq:constructible g}
\end{equation}
for some subanalytic $f,f_{1},...,f_{N}\in\mathcal{C}(\R^{n})$. By
setting $\alpha_{g}(y,r)=\alpha_{g}(y,1)$ for $r>1$, we may assume
$r\leq1$.

There is an open subanalytic subset $U\subseteq\R^{n}$ with complement
$U^{c}$ of measure $0$, such that $f|_{U}$ and all $f_{j}|_{U}$
are continuous (see e.g. \cite[Theorem 3.2.11]{DvdD88}). Set 
\[
M(y):=\max_{j}\left(\max\left\{ \left|f(y)\right|+1,\,2f_{j}(y),\,2/f_{j}(y)\right\} \right),
\]
and denote $h(y):=\frac{M(y)^{-N}}{2N}$. Let $\alpha_{f},\alpha{}_{f_{j}}$
be as in (\ref{eq:radius of approx constancy}) and set 
\[
\alpha_{g}(y,r)=\frac{1}{2}\min\left\{ \alpha_{f}\left(y,rh(y)\right),\alpha_{f_{1}}\left(y,rf_{1}(y)h(y)\right),\dots,\alpha_{f_{N}}\left(y,rf_{N}(y)h(y)\right)\right\} .
\]
Note that $h$,$\alpha_{f_{i}},\alpha_{f_{ij}}$ are subanalytic,
and thus also $\alpha_{g}$ is subanalytic. Moreover, by the continuity
of $f|_{U},f_{j}|_{U}$ we get that $\alpha_{f}(y,\,\cdot\,),\alpha{}_{f_{j}}(y,\,\cdot\,)>0$
and thus also $\alpha_{g}(y,\,\cdot\,)>0$ for all $y\in U$. Note
that for any real numbers $a_{1},...,a_{N},b_{1},...,b_{N}\in[-L,L]$
we have 
\begin{align}
\left|\prod_{i=1}^{N}a_{i}-\prod_{i=1}^{N}b_{i}\right| & =\left|\sum_{j=1}^{N}\left(\prod_{i=1}^{N-j+1}a_{i}\prod_{i=N-j+2}^{N}b_{i}-\prod_{i=1}^{N-j}a_{i}\prod_{i=N-j+1}^{N}b_{i}\right)\right|\nonumber \\
 & \leq\sum_{j=1}^{N}\left|\prod_{i=1}^{N-j}a_{i}\prod_{i=N-j+2}^{N}b_{i}\right|\left|a_{N-j+1}-b_{N-j+1}\right|\leq L^{N-1}\sum_{j=1}^{N}\left|a_{N-j+1}-b_{N-j+1}\right|.\label{eq:difference between products}
\end{align}
By (\ref{eq:radius of approx constancy}), for every $y'\in B(y,\alpha_{g}(y,r))$
and $r<1$ we have $\left|f(y')-f(y)\right|\leq rh(y)\leq\frac{1}{2}$,
and 
\[
\left|f(y')\right|\leq\left|f(y')-f(y)\right|+\left|f(y)\right|\leq M(y).
\]
Similarly, we have: 
\[
\left|\log(f_{j}(y'))-\log(f_{j}(y))\right|\leq\log(1+rh(y))\leq rh(y)\leq\frac{1}{2},
\]
and 
\[
\left|\log(f_{j}(y'))\right|\leq\frac{1}{2}+\left|\log(f_{j}(y))\right|\leq\frac{1}{2}+\max\left(f_{j}(y),\frac{1}{f_{j}(y)}\right)\leq M(y).
\]
By (\ref{eq:difference between products}), for every $y'\in B(y,\alpha_{g}(y,r))$
we have: 
\begin{align*}
\left|g(y)-g(y')\right| & =\left|f(y)\cdot\prod_{j=1}^{N}\log(f_{j}(y))-f(y')\cdot\prod_{j=1}^{N}\log(f_{j}(y'))\right|\\
 & \leq M(y)^{N}\left(\left|f(y)-f(y')\right|+\sum_{j=1}^{N}\left|\log\left(\frac{f_{j}(y)}{f_{j}(y')}\right)\right|\right)\leq rh(y)(N+1)M(y)^{N}\leq r.\qedhere
\end{align*}
\end{proof}
We can now finish the proof of Theorem \ref{thm:characterization of Lq for all q}. 
\begin{proof}[Proof of the Archimedean part of $(1)\Rightarrow(2)$ of Theorem \ref{thm:characterization of Lq for all q}]
Let $\mu\in\mathcal{CM}_{c,\infty}(X(\R))$ and write $g(y)\in\mathcal{C}(\R^{m})$
for the density of $\varphi_{*}\mu$ with respect to $\mu_{\R}^{m}$.
Let $\alpha_{g}$ be as in Lemma \ref{Lemma:equals avarage}, and
set 
\[
S(2):=\left\{ y\in\R^{m}:\frac{1}{2}\leq\alpha_{g}(y,\frac{1}{2})\text{ and }g(y)\neq0\right\} .
\]
Then for each $r\in\R_{\geq3}$ define the following subanalytic set:
\[
S(r):=\left\{ y\in\R^{m}:\frac{1}{r}\leq\alpha_{g}(y,\frac{1}{2})<\frac{1}{r-1}\text{ and }g(y)\neq0\right\} .
\]
We fix $L\in\N$ large enough. Setting 
\begin{equation}
G(y,r)=\frac{(\phi_{*}\mu)(B(y,r))}{r^{m}}=\frac{1}{r^{m}}\int_{B(y,r)}g(y')dy'\label{eq:avaraging of g''}
\end{equation}
and using Lemma \ref{Lemma:equals avarage}, Proposition \ref{Prop:logarithmic explosion-Archimedean}
and H\"older's inequality, we have: 
\begin{align}
\int_{S(r)}g(y)^{p}dy & \leq\mu_{\R}^{m}(S(r))+\int_{S(r)\cap\{g(y)>1\}}g(y)^{p}dy\lesssim1+\int_{\R^{m}}1_{S(r)}(y)\cdot G(y,\frac{1}{r})^{p}dy\nonumber \\
 & \lesssim1+\mu_{\R}^{m}(S(r))^{\frac{1}{1+\frac{1}{L}}}\left(\int_{\R^{m}}G(y,\frac{1}{r})^{(L+1)p}dy\right)^{\frac{1}{L+1}}\nonumber \\
 & \lesssim1+\mu_{\R}^{m}(S(r))^{\frac{1}{1+\frac{1}{L}}}\left|\log\left(\frac{1}{r}\right)\right|^{\frac{M_{\mu,(L+1)p}}{L+1}}=1+\mu_{\R}^{m}(S(r))^{\frac{1}{1+\frac{1}{L}}}\left|\log(r)\right|^{\frac{M_{\mu,(L+1)p}}{L+1}}.\label{eq:g(y)^p}
\end{align}
Since $\mu_{\R}^{m}(S(r))$ is a constructible function, using a similar
argument as in the Proposition \ref{Prop:logarithmic explosion-Archimedean},
and writing $\mu_{\R}^{m}(S(1/r))$ as in (\ref{eq:rewriting G_p(r)})
and (\ref{eq:assymptotic expansion of constructible functions}),
we get $\mu_{\R}^{m}(S(r))\sim r^{\beta}\left|\log(r)\right|^{\gamma}$,
as $r\rightarrow\infty$, for $\beta\in\Q$ and $\gamma\in\N$. But
since $\sum_{r=2}^{\infty}\mu_{\R}^{m}(S(r))<\infty$, we must have
$\beta<-1$. In particular, taking $L$ large enough, we get that
$\mu_{\R}^{m}(S(r))^{\frac{1}{1+\frac{1}{L}}}\sim r^{-1-\delta}\left|\log(r)\right|^{\gamma'}$
for some $\delta>0$. Thus, the following holds for any $\gamma''\in\R$:
\begin{equation}
\sum_{r=2}^{\infty}\mu_{\R}^{m}(S(r))^{\frac{1}{1+\frac{1}{L}}}\left|\log(r)\right|^{\gamma''}<\infty.\label{eq:size of S(r)}
\end{equation}
By (\ref{eq:g(y)^p}) and (\ref{eq:size of S(r)}), we get 
\[
\int_{\R^{m}}g(y)^{p}dy=\sum_{r=2}^{\infty}\int_{S(r)}g(y)^{p}dy\lesssim1+\sum_{r=3}^{\infty}\mu_{\R}^{m}(S(r))^{\frac{1}{1+\frac{1}{L}}}\left|\log(r)\right|^{\frac{M_{\mu,(L+1)p}}{L+1}}<\infty.\qedhere
\]
\end{proof}
\bibliographystyle{alpha}
\bibliography{bibfile}

\end{document}